\documentclass[12pt]{article}

\usepackage[utf8]{inputenc}
\usepackage{titlesec}
\usepackage{graphicx}
\usepackage[a4paper,width=150mm,top=25mm,bottom=25mm,bindingoffset=6mm]{geometry}
\usepackage{fancyhdr}
\usepackage[english]{babel}
\usepackage[numbers]{natbib}
\usepackage{url}
\usepackage[hidelinks]{hyperref}
\usepackage[nodayofweek]{datetime}
\usepackage[all,cmtip]{xy}
\usepackage{comment}

\usepackage{physics}
\usepackage{amsthm}
\usepackage{mathtools}
\usepackage{amssymb}
\usepackage{bbm}
\usepackage[shortlabels]{enumitem}
\usepackage{tikz-cd}
\usepackage{datetime}
\usepackage{amsmath}
\usepackage{calligra}

\numberwithin{equation}{section}

\theoremstyle{definition}
\newtheorem{definition}{Definition}[section]
\theoremstyle{definition}

\theoremstyle{definition}
\newtheorem{remark}{Remark}[section]
\theoremstyle{remark}

\theoremstyle{plain}
\newtheorem{prop}{Proposition}[section]

\theoremstyle{plain}
\newtheorem{theorem}{Theorem}[section]
\newtheorem*{theorem*}{Theorem}
\theoremstyle{plain}
\newtheorem{cor}{Corollary}[section]
\theoremstyle{plain}
\newtheorem{lemma}{Lemma}[section]

\theoremstyle{plain}
\newtheorem*{thm}{Main Theorem}

\theoremstyle{remark}
\newtheorem*{setup}{Set-up}
\theoremstyle{definition}

\newtheorem*{results}{Main results}
\newtheorem*{outline}{Outline of the paper}
\newtheorem*{aknowledgments}{Aknowledgments}

\newcommand{\N}{\mathbb{N}}

\newcommand{\CP}{\mathbb{P}}

\newcommand{\C}{\mathbb{C}}
\newcommand{\Ha}{\mathbb{H}}
\newcommand{\A}{\mathbb{A}}

\newcommand{\sheaf}[1]{\mathcal{#1}}

\newcommand{\mf}{\mathcal{M}}
\newcommand{\Mss}{M^{\mu ss}}

\newcommand{\grothgrp}[1]{K(#1)}
\newcommand{\SL}[1]{\textup{SL}(#1)}

\newcommand{\grothnum}[1]{\grothgrp{#1}_\text{num}}
\newcommand{\schC}{(\text{Sch}^{ft}/\C)}
\newcommand{\schCW}{(\text{Sch}^{wn}/\C)}
\newcommand{\Rss}{R^{ss}}

\newcommand{\Sym}[1]{\text{Sym}#1}

\DeclareMathOperator{\Pic}{Pic}
\DeclareMathOperator{\Coker}{Coker}

\DeclareMathOperator{\Hom}{Hom}
\DeclareMathOperator{\sheafHom}{\sheaf{H}\!\textit{om}\,}

\DeclareMathOperator{\Ext}{Ext}

\DeclareMathOperator{\rk}{rk}
\DeclareMathOperator{\Exts}{\mathcal{E}\!\mathit{xt}}
\DeclareMathOperator{\depth}{depth}

\DeclareMathOperator{\length}{length}
\DeclareMathOperator{\codim}{codim}

\DeclareMathOperator{\Ass}{Ass}
\DeclareMathOperator{\Spec}{Spec}
\DeclareMathOperator{\Quot}{Quot}
\DeclareMathOperator{\Hilb}{Hilb}
\DeclareMathOperator{\Supp}{Supp}
\DeclareMathOperator{\Grass}{Grass}
\DeclareMathOperator{\Coh}{Coh}

\DeclareMathOperator{\gr}{gr}

\newcommand{\Address}{{
  \bigskip
  \footnotesize

  \textsc{Universit\'e de Lorraine, IECL, F-54000 Nancy, France}\par\nopagebreak
  \textit{E-mail address}: \texttt{mihai-cosmin.pavel@univ-lorraine.fr}
}}

\begin{document}

\title{Moduli spaces of slope-semistable pure sheaves}
\author{Mihai Pavel}
\date{}

\maketitle
\begin{abstract}
    We construct a moduli space of slope-semistable pure sheaves, building upon previous work of Le Potier and Jun Li on torsion-free sheaves over smooth surfaces. In particular, our construction provides a compactification of the Simpson moduli space of slope-stable reflexive sheaves. We also prove an effective restriction theorem for slope-(semi)stable pure sheaves following an approach due to Langer.  
\end{abstract}
\section{Introduction}

Moduli spaces of sheaves have been extensively researched throughout the last decades, especially over low-dimensional base schemes. Their study spans across a large variety of fields, with important applications in hyperkähler geometry, gauge theory, enumerative geometry, etc. The primary point of departure concerns the construction of a moduli space of sheaves in the category of schemes. In this regard, one has to restrict to a bounded family of sheaves, in order to ensure that the moduli space will be at least of finite type over the base field. This is usually done by imposing certain stability conditions on sheaves, such as slope-stability or Gieseker-stability. The former condition was first introduced by Mumford \cite{mumford63} while studying the moduli of sheaves over curves, and later generalized to higher dimensions by Takemoto \cite{takemoto1972stable}. This notion is naturally defined for torsion-free sheaves over smooth varieties, although it has a version, called $\hat{\mu}$-stability, well defined also for pure sheaves that might be supported in positive codimension. We shall introduce $\hat{\mu}$-semistability in Section~\ref{section:Preliminaries}, but for a more detailed account of these concepts and their basic properties we refer the reader to \cite[Section~1]{huybrechts2010geometry}.

The precise formulation of the moduli problem is usually done using the categorical language of (co)representable functors. Let $X$ be a smooth projective variety over $\C$ together with a fixed polarization $\sheaf{O}_X(1)$. Choose a numerical Grothendieck class $c \in \grothnum{X}$, that will fix the topological type of sheaves, and consider the moduli functor
\begin{align*}
\mf : \schC ^\text{op} \to (\text{Sets})
\end{align*}
that associates to any scheme $S$ of finite type over $\C$ the set of all equivalence classes of $S$-flat families of slope-semistable sheaves of class $c$ on $X$. In this language, the existence of a fine moduli space of sheaves is equivalent to the respresentability of the functor $\mf$. In general, as semistable sheaves have many automorphisms, we should not expect that $\mf$ is represented. However, one may still investigate if the functor $\mf$ is corepresented. For this to happen, one has to work instead with the notion of Gieseker-semistability, a stronger condition than slope-semistability, that was used by Gieseker \cite{gieseker77} to construct projective moduli spaces of torsion-free sheaves over smooth surfaces. His methods were further generalized by Maruyama \cite{maruyama77} in higher dimensions. In addition, with a slightly different approach, Simpson \cite{simpson1994moduli} was later able to work-out the construction in the case of pure sheaves. It is noteworthy that all these constructions are based on the well-known GIT technique developed by Mumford \cite{mumfordGIT}.

From an algebraic point of view, working with Gieseker-stability gives the right moduli space of sheaves with good functorial properties. On the other hand, a strong motivation for studying slope-(semi)stability spurs from the so-called Kobayashi--Hitchin correspondence, which asserts that holomorphic vector bundles over compact complex manifolds admit Hermitian Yang--Mills (HYM) connections if only if they are slope-polystable. It is not suprising that this significant result, also known as the Donaldson--Uhlenbeck--Yau Theorem \cite{donaldson1985anti, uhlenbeck1986existenceHYM}, has found many applications in the literature, as it draws a deep connection between algebraic geometry and gauge theory. In what follows, we present only the algebraic side of the story, but we refer the reader to \cite{friedman1994smooth} for the analytical aspects of the theory. 

In contrast to the case of Gieseker-stability, there are still many open problems concerning the construction of a right moduli space of slope-semistable sheaves. In this respect, Langer \cite{langer2021moduli} has already pointed out several drawbacks of working with slope-semistability, such as the fact that the moduli stack of torsion-free slope-semistable sheaves is neither $S$-complete nor $\Theta$-reductive (in the sense of \cite{alper2019existence}). This indicates that the existence of a moduli space corepresenting the moduli functor $\mf$ in the category of schemes is very unlikely. Building upon work of Le Potier and Jun Li, Huybrechts and Lehn constructed in \cite[Ch.~8]{huybrechts2010geometry} a projective scheme $\Mss$ parametrizing slope-semistable sheaves with fixed determinant on a smooth surface. Even though the scheme $\Mss$ does not corepresent $\mf$, it turns out that $\Mss$ is closely related to the Donaldson--Uhlenbeck compactification of HYM connections (see \cite{li1993algebraic}), a reminder of the close ties with gauge theory. Greb and Toma \cite{greb2017compact} later generalized their construction over higher dimensional base schemes endowed with a multipolarization, but only within the category of weakly normal varieties. Also their moduli spaces over higher-dimensional bases found applications in gauge theory (see \cite{greb2021HYM}). We wish to emphasize that the work mentioned above treats only the case of torsion-free sheaves of fixed determinant on smooth varieties.

\begin{results}
In this paper we aim to construct a moduli space of $\hat{\mu}$-semistable pure sheaves on $X$. As before, the general strategy follows ideas of Le Potier \cite{le1992fibre} and Jun Li \cite{li1993algebraic}, even though our approach differs in several relevant aspects that we carefully point out below: 

\begin{enumerate}[wide, labelwidth=!,nosep]
  \item The key ingredient in all these constructions is the Mehta--Ramanathan restriction theorem, which asserts that slope-(semi)stability is preserved by restriction to a general divisor of sufficiently large degree. This was previously known for torsion-free sheaves over smooth varieties (see \cite{mehta1982semistable, mehta1984restriction,flenner1984restrictions, langer2004semistable}). We prove here \textit{a restriction theorem for pure sheaves}, whose support schemes might be lower dimensional and singular.

\begin{theorem*}
Let $E$ be a $\hat{\mu}$-(semi)stable sheaf on a polarized smooth variety $(X,\sheaf{O}_X(1))$ over $\C$. Then the restriction $E|_D$  to a general divisor $D \in |\sheaf{O}_X(a)|$ remains $\hat{\mu}$-(semi)stable for $a \gg 0$. 
\end{theorem*}
In fact, in Section~\ref{section:RestrictionThms} we prove something more, by also giving an effective bound on the degree of the divisor from which the theorem holds. The methods employed in the proof are based on \cite{langer2004semistable}. Due to the geometric importance of slope semistability, our restriction theorem is of independent interest and might have applications in the study of moduli spaces of pure sheaves supported in positive codimension. Our result also generalizes a theorem of H\"oring and Peternell \cite[Lem.~2.11]{horing2019algebraic}, which they apply to the study of the Beauville-Bogomolov decomposition for minimal models. Their result is the only restriction theorem we are aware of over singular spaces.

\item To ease up the presentation, we write down the construction when $X$ is endowed with a single polarization $\sheaf{O}_X(1)$, but the same techniques could be adapted for a multipolarization. Since the $\hat{\mu}$-semistable sheaves of class $c$ on $X$ form a bounded family, cf. \cite[Thm.~3.3.7]{huybrechts2010geometry}, they all fit inside a locally closed subscheme $\Rss$ of an appropriate Quot scheme (see Section~\ref{section:Construction}). However, any quotient of $\Rss$ is uniquely determined only up to the action of a linear group $G$. We would like to quotient $\Rss$ by $G$, but unfortunately this problem does not admit a good GIT quotient. Nevertheless, $\Rss$ is endowed with a natural $G$-equivariant line bundle $\sheaf{L}$ (see Section~\ref{section:Construction}), and we will show in Theorem~\ref{thm:semiAmpleness} that a power of $\sheaf{L}$ is globally generated by $G$-invariant sections, after replacing $\Rss$ by its weak normalization if necessary. To achieve this, we will restrict the universal family of quotients over $\Rss$ to a smooth complete intersection $C \subset X$ such that its support has fiber dimension one, and where slope-(semi)stability coincides with Gieseker-(semi)stability. Through this process, for a fixed quotient $[E] \in \Rss$, our restriction theorem ensures that its $\hat{\mu}$-semistability is preserved if we choose $C$ sufficiently general of large enough degree. Hence $E|_C$ lies inside the Simpson moduli space $M^{Gss}_C$ of $1$-dimensional (Gieseker-)semistable sheaves on $C$ \cite{simpson1994moduli}. This will be enough to construct a $G$-invariant rational map $\psi_E: \Rss \dashrightarrow M^{Gss}_C$, well-defined around $[E]$, such that some power $\sheaf{L}^\nu$ is the pull-back via $\psi_E$ of an ample line bundle on $M^{Gss}_C$. This way we will obtain a local $G$-invariant section of $\sheaf{L}^\nu$ non-vanishing at $[E]$, which moreover will be defined outside a closed subset of codimension $\geq 2$ in $\Rss$ (see Section~\ref{subsect:semiamplenessTheorems} for the details). Then, by using the weak normality assumption as in \cite{greb2017compact}, we will extend this section over the whole $\Rss$. Repeating this construction around each point of $\Rss$ will show that a power of $\sheaf{L}$ is globally generated by $G$-invariant sections. 

In \cite{huybrechts2010geometry} and \cite{greb2017compact}, the authors follow the same lines as above, except that they use the Gieseker--Maruyama construction of $M^{Gss}_C$, which endows $M^{Gss}_C$ with a relative ample line bundle $\sheaf{A}$ over $\Pic(C)$. In order to separate points of $\Rss$, they pull back sections (of some power) of $\sheaf{A}$ from $M^{Gss}_C$ to $\Rss$, and for this reason they need to fix the determinant of sheaves. Our approach overcomes this restriction, we \textit{do not fix the determinant} since we use an (absolute) ample line bundle over $M^{Gss}_C$ in the corresponding case, as given by Simpson's construction. Also, one can immediately derive from our methods a relative moduli space over $\Pic(X)$. 
  \item To end the construction, one then considers the Proj-scheme of a $\C$-algebra 
\begin{align*}
  \bigoplus_{k \geq 0} H^0(\Rss,\sheaf{L}^{k\nu})^G
\end{align*}
for some power $\nu > 0$ such that $\sheaf{L}^\nu$ is globally generated, and declares it the moduli space of slope-semistable sheaves. However, as $\Rss$ is not necessarily a proper scheme, some care has to be taken since this algebra is not a priori of finite type over $\C$. In contrast to \cite{huybrechts2010geometry, greb2017compact}, we follow a different path for proving the projectivity of the moduli space, since we are under the impression that there is a missing step in \cite[Prop.~8.2.6]{huybrechts2010geometry} (see Remark~\ref{remark:HL}). In doing so, we first introduce the notion of $G$-properness (see Definition~\ref{def:Gproper}) for $G$-invariants morphisms of schemes, which generalizes the classical notion of properness when $G$ acts trivially. This holds in particular for $\Rss$ (see Lemma~\ref{lemma:RssGproper}), as a consequence of the valuative criterion of properness due to Langton \cite{langton1975valuative}. We further define what we call the relative $G$-normalization, an equivariant version of the well-known Stein factorization, and then show its existence when $G$ is reductive (see Proposition~\ref{Prop:normalization}). Using this ingredient, we construct \textit{the Iitaka $G$-fibration} corresponding to a $G$-proper scheme endowed with a $G$-equivariant semiample line bundle (see Theorem~\ref{thm:Unique}). Our work generalizes the classical Iitaka fibration to the equivariant case, and may find applications in other places as well.  

Finally, we employ the Iitaka $G$-fibration to construct the moduli space. As a culmination of our work, we obtain the following main result proving the existence and uniqueness of a moduli space parametrizing $\hat{\mu}$-semistable sheaves of class $c$ on $X$. We denote by $\mf^{wn}$ the restriction of $\mf$ to the full subcategory $\schCW$ of weakly normal varieties. 

\begin{thm}
Let $X$ be a smooth projective variety over $\C$ together with a fixed polarization $\sheaf{O}_X(1)$, and choose a numerical Grothendieck class $c \in \grothnum{X}$ of dimension $d > 1$. Then there exists a unique triple $(\Mss, \sheaf{A}, e)$ formed of a weakly normal projective variety $\Mss$ endowed with an ample line bundle $\sheaf{A}$ and a natural number $e > 0$ such that there is a natural transformation $\Psi: \mf^{wn} \to \Hom_{\schCW}(-,\Mss)$, that associates to any weakly normal variety $S$ and any $S$-flat family $\sheaf{E}$ of $\hat{\mu}$-semistable sheaves of class $c$ on $X$ a classifying morphism $\Psi_\sheaf{E}: S \to \Mss$, satisfying the following properties:
    \begin{enumerate}[(1)]
\item For any $S$-flat family $\sheaf{E}$ of $\hat{\mu}$-semistable sheaves of class $c$ on $X$, the classifying morphism $\Psi_\sheaf{E}$ satisfies
\begin{align*}
    \Psi_\sheaf{E}^*(\sheaf{A}) \cong \lambda_\sheaf{E}(w_{l',m'} \cdot h^{d-1})^e,
\end{align*}
where $\lambda_\sheaf{E}(w_{l',m'} \cdot h^{d-1})$ is the determinant line bundle on $S$ defined in Section~\ref{section:LineBundles}.
\item For any other triple $(M',\sheaf{A}',e')$, with $M'$ a projective scheme over $\C$, $\sheaf{A}'$ an ample line bundle on $M'$ and $e'$ a natural number satisfying property (1), we have $e | e'$ and there exists a unique morphism $\phi: \Mss \to M'$ such that $\phi^*\sheaf{A}' \cong \sheaf{A}^{(e'/e)}$.
\end{enumerate}
Moreover, if the class $c$ has dimension $2$, then the same result as above holds true over the category $\schC$ of schemes of finite type over $\C$.
\end{thm}

\item In the last section we describe the geometric points of the moduli space $\Mss$. Consequently, we generalize several known results in the literature to the case of pure sheaves (see \cite{le1992fibre, li1993algebraic, huybrechts2010geometry, greb2017compact}). For a detailed exposition of the state-of-the-art in the torsion-free case see also \cite[Sect.~2.8]{greb2021HYM}. In particular, we show in Proposition~\ref{prop:compactification} that our moduli space $\Mss$ provides a compactification of the Simpson moduli space of slope-stable reflexive (pure) sheaves. We summarize the separation results in Corollary~\ref{cor:resultsep}.
\end{enumerate}
\end{results}

\begin{outline}
In Preliminaries, we introduce the notion of $\hat{\mu}$-stability and recall some of its basic properties. Here we also define the reflexive hull of a pure sheaf, that will be needed later for describing the geometric points of the moduli space. In Section~\ref{section:RestrictionThms}, we prove our restriction theorems for slope-semistability in the pure case, by following a strategy developed by Langer. In Section~\ref{section:LineBundles}, we introduce a class of determinant line bundles which correspond functorially to any $S$-flat family of slope-semistable sheaves and show their semiampleness, as stated by Theorem~\ref{thm:semiAmpleness}. Similarly to the case studied in \cite{greb2017compact}, in order to prove the semiampleness result in higher dimensions we need to work under weak normality assumptions on schemes. In Section~\ref{sect:Gproperness} we generalize the Stein factorization and the Iitaka fibration to the equivariant case. For this, we introduce the notion of $G$-properness for $G$-invariant morphisms and describe some of its basic properties. Among other things, here we also prove a devissage lemma for equivariant sheaves on $G$-proper schemes. Finally, in Section~\ref{section:Construction} we give the construction of the moduli space of sheaves. We end by studying the geometric points of $\Mss$.
\end{outline}

\begin{aknowledgments}
The author wishes to express his gratitude to his supervisor Matei Toma for his patient guidance and fruitful discussions during the preparation of this paper, which will be part of the author's PhD thesis. The author wishes to thank Adrian Langer for his helpful comments on a preliminary version of this paper.
\end{aknowledgments}
\section{Preliminaries}\label{section:Preliminaries}

In this section we set up notation and terminology. We first introduce the slope-semistability conditions and then recall some known facts about the reflexive hull of a pure sheaf. The main reference here is \cite[Sect.~1]{huybrechts2010geometry}.   

Throughout this paper, let $X$ be a connected smooth projective variety of dimension $n$ over $\C$, and fix $\sheaf{O}_X(1)$ a very ample line bundle on $X$. 

Let $E$ be a coherent sheaf of pure dimension $d$ on $X$, i.e. $\dim(F) = d$ for all non-zero subsheaves $F \subset E$. The Hilbert polynomial of $E$ is given by $P(E,m) = h^0(X,E(m))$ for $m \gg 0$, and has the form
\begin{align*}
    P(E,m) = \sum_{i=0}^{d} \alpha_i(E)\frac{m^i}{i!},
\end{align*}
with rational coefficients $\alpha_i(E)$. We call $r(E) \coloneqq \alpha_d(E)$ the \textit{multiplicity} of $E$.

If $E$ is pure of dimension $d = \dim(X)$, also called \textit{torsion-free}, define its \textit{rank} by $\rk(E)\coloneqq \alpha_d(E)/\alpha_d(\sheaf{O}_X)$ and its $\mu$-\textit{slope} by $\mu(E) \coloneqq c_1(E)H^{d-1}/\rk(E)$ for some divisor $H \in |\sheaf{O}_X(1)|$. Recall that $E$ is called $\mu$-\textit{semistable} (resp.~$\mu$-\textit{stable}) if $\mu(F) \leq \mu(E)$ (resp.~$<$) for all subsheaves $F \subset E$ of rank $0 < \rk(F) < \rk(E)$.

A simple computation using the Hirzebruch-Riemann-Roch formula yields
\begin{align*}
    c_1(E)H^{d-1} = \alpha_{d-1}(E)-\rk(E)\alpha_{d-1}(\sheaf{O}_X).
\end{align*}
Hence one may also use the ratio $\hat{\mu}(E) \coloneqq \alpha_{d-1}(E)/\alpha_d(E)$, called the $\hat{\mu}$-\textit{slope} of $E$, to check the slope-semistability of $E$. Indeed, as we have
\begin{align*}
  \mu(E) = \alpha_d(\sheaf{O}_X)\hat{\mu}(E)-\alpha_{d-1}(\sheaf{O}_X).
\end{align*}
This motivates the following definition of slope-semistability for pure sheaves:

\begin{definition}
A coherent sheaf $E$ of dimension $d$ is called $\hat{\mu}$-\textit{semistable} (resp.~$\hat{\mu}$-\textit{stable}) if $E$ is pure and $\hat{\mu}(F) \leq \hat{\mu}(E)$ (resp.~$<$) for all subsheaves $F \subset E$ of multiplicity $0 < \alpha_d(F) < \alpha_d(E)$.
\end{definition}

If $E$ is pure, then $E$ has a unique Harder-Narasimhan filtration
\begin{align*}
    0 = E_0 \subset E_1 \subset \ldots \subset E_l = E
\end{align*}
such that its factors $E_i/E_{i-1}$ are $\hat{\mu}$-semistable and
\begin{align*}
    \hat{\mu}(E_1) > \hat{\mu}(E_2/E_1) > \ldots > \hat{\mu}(E/E_{l-1}).
\end{align*}
We denote $\hat{\mu}_{\max}(E) \coloneqq \hat{\mu}(E_1)$ and $\hat{\mu}_{\min}(E) \coloneqq \hat{\mu}(E/E_{l-1})$. If $E$ is already $\hat{\mu}$-semistable, then $E$ has a Jordan-H\"{o}lder filtration
\begin{align*}
    0 = E_0 \subset E_1 \subset \ldots \subset E_l = E
\end{align*}
such that its factors $E_i/E_{i-1}$ satisfy $\hat{\mu}(E_i/E_{i-1}) = \hat{\mu}(E)$. In this case the filtration is not unique, but nevertheless its graded sheaf $\gr_{JH}(E) \coloneqq \bigoplus_i E_i/E_{i-1}$ is so in codimension one, i.e. outside a closed subset of codimension $\geq 2$ in $\Supp(E)$.

\subsection{Reflexive hull of a sheaf}\label{subsect:reflHull}

Let $E$ be a pure sheaf on $X$. Recall that $E$ satisfies property $S_2$ if
\begin{align*}
  \depth_x(E) \geq \min\{2, \codim(x,\Supp(E))\} \text{ for all }x \in \Supp(E). 
\end{align*}
If $E$ is torsion-free, then $E$ has a unique reflexive hull $E^{\vee \vee} \coloneqq \sheafHom(\sheafHom(E,\sheaf{O}_X),\sheaf{O}_X)$ that is also torsion-free and satisfies $S_2$. If $E$ is pure of dimension $d$, we set
\begin{align*}
 E^H \coloneqq \Exts^{n-d}(\Exts^{n-d}(E,\omega_X),\omega_X),
\end{align*}
where $\omega_X$ is the dualizing line bundle on $X$. We call $E^H$ the \textit{reflexive hull} of $E$. Then there is a natural inclusion $E \to E^H$, cf. \cite[Prop.~1.1.10]{huybrechts2010geometry}, such that $E^H/E$ has codimension $\geq 2$ in $\Supp(E)$. Moreover $E^H$ is pure and satisfies $S_2$. 

Notice that we recover the classical notion for $d=n$. Also, this definition agrees with the more general one given by Koll{\'a}r in \cite[Sect.~9]{kollar2017families}. For a throughout study of reflexive hulls we refer the reader to his book.

\section{Restriction theorems}\label{section:RestrictionThms}

In this section we prove restriction theorems for $\hat{\mu}$-semistable pure sheaves. When the base scheme is smooth, it is already known that slope-semistability or slope-stability is preserved by restriction to a general divisor of sufficiently large degree. The first general results of this kind were proven by Mehta and Ramanathan in \cite{mehta1982semistable, mehta1984restriction}. However their restriction theorems are not effective, as they provide no bound on the degree of the divisor. Such an effective restriction theorem, but only for slope-semistability, was later given by Flenner \cite{flenner1984restrictions} in zero characteristic. Finally, Langer \cite{langer2004semistable} was able to prove stronger effective restriction theorems that also work in positive characteristic. Here we generalize his method to the case of pure sheaves supported in positive codimension.

We end this section by studying the preservation of flatness for families of sheaves under restriction to divisors.

\subsection{Restriction theorems for $\hat{\mu}$-semistability}

Let $E$ be a purely $d$-dimensional sheaf on $X$. Denote its support by $Y$ and let $\sheaf{O}_Y(1)$ be the pullback of $\sheaf{O}_X(1)$ to $Y$. One can find a linear subspace $W \subset H^0(Y, \sheaf{O}_Y(1))$ of dimension $d+1$ that gives a surjective finite morphism $\pi: Y \to \CP(W)$ with $\pi^*(\sheaf{O}(1))=\sheaf{O}_Y(1)$. Here $\sheaf{O}(1)$ is the canonical line bundle corresponding to a hyperplane $H \in \CP(W^\vee)$. Since $\pi_*$ is exact, $E$ and $\pi_* E$ have the same Hilbert polynomial. Thus, $\alpha_d(E) = \rk(\pi_* E)$ and
\begin{align*}
    \hat{\mu}(E) = \mu(\pi_* E) + \hat{\mu}(\sheaf{O}_{\CP(W)}) = \mu(\pi_* E) + (d+1)/2,
\end{align*}
where the slope $\mu(\pi_* E)$ is computed with respect to $H$. Note that semistability is not necessarily preserved by pushforward, but we have the following inequalities due to Langer \cite[Lem.~6.2.2]{langer2006moduli}.
\begin{lemma}\label{lemma:Langerpushforward}
Under the above notation, we have
\begin{enumerate}[nolistsep]
    \item $\mu_{\max}(\pi_* E) - \mu(\pi_* E) \leq \hat{\mu}_{\max}(E) - \hat{\mu}(E) + \rk(\pi_* E)^2$,
    \item $\mu(\pi_* E) - \mu_{\min}(\pi_* E) \leq \hat{\mu}(E) - \hat{\mu}_{\min}(E) + \rk(\pi_* E)^2$.
\end{enumerate}
\end{lemma}

For $a > 0$, denote by $\Ha_a = \CP(S^a(W^\vee))$ the linear system of hypersurfaces of degree $a$ in $\CP(W)$. For $D \in \Ha_a$, the scheme-theoretic inverse image $\pi^{-1}(D)$ is an effective divisor on $Y$, which is $E$-regular and moreover satisfies $\pi_*(E|_{\pi^{-1}(D)}) = \pi_*(E)|_D$. If the restriction $E|_{\pi^{-1}(D)}$ is pure, one can easily check that
\begin{align}\label{eq:compareSlopes}
    \hat{\mu}(E|_{\pi^{-1}(D)}) = \hat{\mu}(\pi_*(E)|_D) = \mu(\pi_*(E)|_D)/a + \hat{\mu}(\sheaf{O}_D).
\end{align}
Note that since $\pi$ is finite, $E|_{\pi^{-1}(D)}$ is pure if and only if $\pi_*(E)|_D$ is pure. 

\begin{lemma}
The restriction $E|_{\pi^{-1}(D)}$ is pure for a general divisor $D \in \Ha_a$.
\end{lemma}
\begin{proof}
By \cite[Cor.~1.1.14 ii)]{huybrechts2010geometry}, if $\pi^{-1}(D)$ contains none of the associated points of $E$ and  $\Exts_{X}^q(E,\omega_X)$ for all $q \geq 0$, then $E|_{\pi^{-1}(D)}$ is pure. Obviously we can choose such a divisor since $W$ is base-point-free.
\end{proof}

In what follows, for any torsion-free sheaf $F$ on $\CP(W)$ the \textit{discriminant} of $F$ is defined as usual by
\begin{align*}
    \Delta(F) = 2 \rk(F) c_2(F) - (\rk(F)-1)c_1(F)^2.
\end{align*}

\begin{lemma}\label{lemma:DeltaG}
Let $G$ be a pure sheaf of dimension $d$ and multiplicity $r \coloneqq \rk(\pi_* G)$ on $Y$. Then
\begin{align*}
    \Delta(\pi_* G)H^{d-2} + r^2(\hat{\mu}_{\max}(G) - \hat{\mu}(G))(\hat{\mu}(G) - \hat{\mu}_{\min}(G)) \geq -r^6.
\end{align*}
\end{lemma}
\begin{proof}
We first introduce some notation. Let $0 = G_0 \subset G_1 \subset \ldots \subset G_l = G$ be the Harder-Narasimhan filtration of $G$. Since $\pi_*$ is exact, $0 = \pi_* G_0 \subset \pi_* G_1 \subset \ldots \subset \pi_* G_l = \pi_* G$ remains a filtration. Set $F_i \coloneqq \pi_* (G_i/G_{i-1})$, $r_i \coloneqq \rk(F_i)$, $\mu_i \coloneqq \mu(F_i)$, $\mu \coloneqq \mu(\pi_* G)$. 

A direct computation using the Hodge Index Theorem gives 
\begin{align*}
    \frac{\Delta(\pi_* G)H^{d-2}}{r} = &\sum \frac{\Delta(F_i)H^{d-2}}{r_i} \\
    &-\frac{1}{r}\sum_{i < j} r_i r_j\left( \frac{c_1(F_i)}{r_i} -  \frac{c_1(F_j)}{r_j}\right)^2H^{d-2}\\
    \geq &\sum \frac{\Delta(F_i)H^{d-2}}{r_i} - \frac{1}{r}\sum_{i < j} r_i r_j( \mu_i -  \mu_j)^2.
\end{align*}
By using \cite[Lem.~1.4]{langer2004semistable}, we get
\begin{align*}
    \Delta(\pi_* G)H^{d-2} + r^2(\mu_1 - \mu)(\mu - \mu_l) \geq \sum \frac{r}{r_i}\Delta(F_i)H^{d-2},
\end{align*}
or equivalently
\begin{align*}
    \Delta(\pi_* G)H^{d-2} + r^2(\hat{\mu}_{\max}(G) - \hat{\mu}(G))(\hat{\mu}(G) - \hat{\mu}_{\min}(G)) \geq \sum \frac{r}{r_i}\Delta(F_i)H^{d-2}.
\end{align*}
 Now the result follows from the lemma bellow.  
\end{proof}

\begin{lemma}\label{lemma:deltaFi}
Under the above assumptions, we have
\begin{align*}
    \sum \frac{r}{r_i}\Delta(F_i)H^{d-2} \geq -r^6.
\end{align*}
\end{lemma}
\begin{proof}
Applying \cite[Thm.~5.1]{langer2004semistable} to each $F_i$, we obtain
\begin{align*}
    \Delta(F_i)H^{d-2} + r_i^2(\mu_{\max}(F_i) - \mu(F_i))(\mu(F_i) - \mu_{\min}(F_i)) \geq 0.
\end{align*}
By Lemma~\ref{lemma:Langerpushforward}, we have
\begin{align*}
    \mu_{\max}(F_i) - \mu(F_i) \leq r_i^2
\end{align*}
and
\begin{align*}
    \mu(F_i) - \mu_{\min}(F_i) \leq r_i^2.
\end{align*}
Putting together the above inequalities yields
\begin{align*}
    \sum \frac{r}{r_i}\Delta(F_i)H^{d-2} \geq - r\sum r_i^5 \geq - r\left(\sum r_i\right)^5 \geq -r^6. 
\end{align*}
\end{proof}

We are now ready to prove the announced restriction theorem. We give a proof by contradiction following the idea of Langer \cite[Thm.~5.2]{langer2004semistable}. In this regard, for the reader's convenience, we adopt his notation.

\begin{theorem}\label{thm:StableRestriction}
Suppose that $E$ is $\hat{\mu}$-stable, and let $D \in \Ha_a$ be a normal divisor on $\CP(W)$ such that $E|_{\pi^{-1}(D)}$ is pure. If 
\begin{align*}
    a > \frac{r-1}{r}\Delta(\pi_* E)H^{d-2} + \frac{1}{r(r-1)} + r^5(r-1),
\end{align*}
then the restriction $E|_{\pi^{-1}(D)}$ is $\hat{\mu}$-stable.
\end{theorem}
\begin{proof}
Suppose, on contrary, that $E|_{\pi^{-1}(D)}$ is not stable. Let $S$ be a saturated destabilizing subsheaf of $E|_{\pi^{-1}(D)}$, and denote $T \coloneqq (E|_{\pi^{-1}(D)}/S)$. Let $G$ be the kernel of the composition $E \to E|_{\pi^{-1}(D)} \to T$. Then there are two short exact sequences
\begin{align*}
    0 \to G \to E \to T \to 0
\end{align*}
and
\begin{align*}
    0 \to E(-\pi^{-1}(D)) \to G \to S \to 0.
\end{align*}
Set $r \coloneqq \rk(\pi_* G)$ and $\rho \coloneqq \rk(\pi_* S)$. A direct computation gives
\begin{align*}
    \Delta(\pi_* G)H^{d-2} = ~&\Delta(\pi_* E)H^{d-2} - \rho(r-\rho)a^2\\
                         &+ 2(rc_1(\pi_* T) - (r-\rho)Dc_1(\pi_*E))H^{d-2}.
\end{align*}
Since $T$ is a destabilizing quotient of $E|_{\pi^{-1}(D)}$, we have $\hat{\mu}(T) \leq \hat{\mu}(E|_{\pi^{-1}(D)})$, or equivalently $\mu(\pi_* T) \leq \mu(\pi_*(E)|_D)$ by \eqref{eq:compareSlopes}. Then $2(rc_1(\pi_* T) - (r-\rho)Dc_1(\pi_*E))H^{d-2} \leq 0$, and so
\begin{align*}
    \Delta(\pi_* G)H^{d-2} \leq \Delta(\pi_* E)H^{d-2} - \rho(r-\rho)a^2.
\end{align*}
By using Lemma~\ref{lemma:DeltaG}, we deduce
\begin{align}\label{ineq:Langton1}
    \Delta(\pi_* E)H^{d-2} - \rho(r-\rho)a^2 + r^2(\hat{\mu}_{\max}(G) - \hat{\mu}(G))(\hat{\mu}(G) - \hat{\mu}_{\min}(G)) \geq -r^6.
\end{align}
Since $E$ and $E(-\pi^{-1}(D))$ are stable, we obtain
\begin{align}\label{ineq:Langton2}
    \hat{\mu}_{\max}(G) - \hat{\mu}(G) = \hat{\mu}_{\max}(G) - \hat{\mu}(E) + \frac{r-\rho}{r}a \leq \frac{r-\rho}{r}a - \frac{1}{r(r-1)}
\end{align}
and
\begin{align}\label{ineq:Langton3}
    \hat{\mu}(G) - \hat{\mu}_{\min}(G) = \hat{\mu}(E) - \hat{\mu}_{\min}(G) + \frac{\rho}{r}a \leq \frac{\rho}{r}a - \frac{1}{r(r-1)}.
\end{align}
Putting together the inequalities \eqref{ineq:Langton1}--\eqref{ineq:Langton3} yields
\begin{align*}
    \Delta(\pi_* E)H^{d-2} - \rho(r-\rho)a^2 + r^2\left(\frac{r-\rho}{r}a - \frac{1}{r(r-1)}\right)\left(\frac{\rho}{r}a - \frac{1}{r(r-1)}\right) \geq -r^6.
\end{align*}
Thus
\begin{align*}
    \Delta(\pi_* E)H^{d-2} - \frac{r}{r-1}a + \frac{1}{(r-1)^2} \geq -r^6,
\end{align*}
which gives a contradiction.
\end{proof}

Note that the theorem above is effective, i.e.~we can choose $a > 0$ large enough to work for all stable sheaves. Furthermore, essentially the same proof gives also an effective restriction theorem for $\hat{\mu}$-semistable sheaves.

\begin{theorem}\label{thm:RestrictionSemi}
Suppose that $E$ is $\hat{\mu}$-semistable. For $a \gg 0$ the restriction $E|_{\pi^{-1}(D)}$ remains $\hat{\mu}$-semistable for a general divisor $D \in \Ha_a$. 
\end{theorem}
\begin{proof}
See the proof in \cite[Cor.~5.4]{langer2004semistable}.
\end{proof}

For $a > 0$, let $\Pi_a \coloneqq |\sheaf{O}_X(a)|$ be the linear system of hypersurfaces of degree $a$ on $X$. By using the openness property of slope-semistable sheaves in flat families, one can further show:

\begin{cor}
Suppose that $E$ is $\hat{\mu}$-(semi)stable. The restriction $E|_D$ to a general divisor $D \in \Pi_a$ remains $\hat{\mu}$-(semi)stable for $a \gg 0$.
\end{cor}
\begin{proof}
We prove the result only when $E$ is $\hat{\mu}$-semistable, as the stable case is completely similar. By Theorem~\ref{thm:RestrictionSemi}, for $a \gg 0$ there is a divisor $D' \in \Ha_a$ such that the restriction $E|_{\pi^{-1}(D')}$ remains $\hat{\mu}$-semistable. Note that the divisor $\pi^{-1}(D')$ must avoid the associated points of $E$.
If we choose $a$ large enough, then $H^0(X, \sheaf{O}_X(a)) \to H^0(Y, \sheaf{O}_Y(a))$ is surjective, and so $\pi^{-1}(D')$ is the restriction of some divisor $D_0 \in \Pi_a$. In particular $D_0$ contains none of the associated points of $E$.

Consider the incidence variety $Z = \{ (D,x) \in \Pi_a \times X \mid x \in D \}$ with its natural projections $p_1$ and $p_2$. If $U \subset \Pi_a$ is the open subset of divisors avoiding the associated points of $E$, then for each $D \in U$ we have a short exact sequence
\begin{align*}
    0 \to E(-D) \to E \to E|_D \to 0.
\end{align*}
Hence the Hilbert polynomial of $E|_D$ remains constant when $D$ varies through $U$, which means that $p_2^*E$ is flat over $U$. But $D_0 \in U$ and since the semistability property is open in flat families (cf.~\cite[Prop.~2.3.1]{huybrechts2010geometry}), we conclude that $E|_D$ remains $\hat{\mu}$-semistable for a general divisor $D \in \Pi_a$.

\end{proof}

\subsection{Restriction of a flat family to a divisor}

Next we turn our attention to the flatness property for families of sheaves. 
In what follows, let $S$ be a $\C$-scheme of finite type and $\sheaf{E}$ a coherent $\sheaf{O}_{S \times X}$-module such that each geometric fiber $\sheaf{E}_s$ has no zero-dimensional associated points.

\begin{lemma}\label{lemma:pureRestriction}
For $a \gg 0$, there exists a general divisor $D \in \Pi_a$ that avoids the associated points of every geometric fiber $\sheaf{E}_s$ for $s \in S$. 
\end{lemma}

\begin{proof}
Consider the family $\Sigma$ of all irreducible closed subschemes of $X$ corresponding to associated points of $\sheaf{E}_s$ for $s \in S$ closed. By \cite[Prop.~1.3]{grothendieck1961techniques}, the family $\Sigma$ is bounded, i.e.~the set of Hilbert polynomials $\{ P_Y \}_{Y \in \Sigma}$ is finite, and for $a \gg 0$ every element of $\Sigma$ is $a$-regular. Without loss of generality, we may assume that all elements of $\Sigma$ have the same Hilbert polynomial, say $P$. In this case, $\Sigma$ is a subset of the Hilbert scheme $\Hilb_X(P)$.

Consider the incidence variety
\begin{align*}
    I = \{ (Y,D) \in \Hilb_X(P) \times \Pi_a \mid Y \in D \}
\end{align*}
with its natural projections $p_1$ and $p_2$. Since each $Y \in \Sigma$ is $a$-regular, its Hilbert polynomial satisfies $P(a) = h^0(\sheaf{O}_{Y}(a))$. Thus, $p_1$ is a projective bundle with fibers of dimension $h^0(\sheaf{O}_X(a))-P(a)$. Therefore, $\dim(I) = \dim(\Hilb_X(P)) + h^0(\sheaf{O}_X(a))-P(a)$. On the other hand, the general fiber of $p_2$ has dimension
\begin{align*}
    \dim(I)-\dim(\Pi_a) = \dim(\Hilb_X(P))-P(a)  < 0.
\end{align*}
The last inequality is satisfied for $a \gg 0$, since $P$ has degree $> 0$ by assumption. In other words, for $a \gg 0$ there is a general divisor $D \in \Pi_a$ which contains no point of $\Hilb_X(P)$. In particular, $D$ contains no element of $\Sigma$.
\end{proof}

\begin{lemma}\label{lemma:flatRestriction}
Assume that $\sheaf{E}$ is $S$-flat and each geometric fiber $\sheaf{E}_s$ has no zero-dimensional associated points. Then the restriction $\sheaf{E}|_{S \times D}$ to a general divisor $D \in \Pi_a$ remains $S$-flat for $a \gg 0$.
\end{lemma}
\begin{proof}
Let $D \in \Pi_a$ be a general divisor satisfying Lemma~\ref{lemma:pureRestriction}. Then $D$ is $\sheaf{E}_s$-regular, i.e. $\sheaf{E}_s(-D) \to \sheaf{E}_s$ is injective, for every $s \in S$. By \cite[Thm.~22.5]{matsumura1989commutative} it follows that $\sheaf{E}|_{S \times D}$ is $S$-flat.
\end{proof}

Even though the family $\sheaf{E}|_{S \times D}$ remains $S$-flat, some of its fibers $\sheaf{E}_s|_D$ may have zero-dimensional associated points. In this case the hypotheses of Lemma~\ref{lemma:flatRestriction} are not fulfilled anymore for $\sheaf{E}|_{S \times D}$, and so we may lose the flatness if we restrict again our family. Nevertheless, we next show that in some cases there exists an open locus $U \subset S$, whose complement has codimension $\geq 2$ in $S$, over which the flatness is preserved.

In what follows, we assume that $S$ is endowed with the action of a connected algebraic group $G$. The following results are a direct generalization to the case of pure sheaves of \cite[Lem.~3.1]{greb2017compact} and \cite[Cor.~3.2]{greb2017compact}. For the reader's convenience, we include their proofs.

\begin{lemma}\label{lemma:nonflatRestriction}
Let $\sheaf{E}$ be a $G$-equivariant $S$-flat family of sheaves on $X$, and let $S_1,\ldots,S_k$ be $G$-stable irreducible closed subschemes of $S$. For each index $i$, suppose there exists a point $s_i \in S_i$ such that $\sheaf{E}_{s_i}$ is pure of dimension $d>0$. Then, for $a \gg 0$, there exists an open subset $U \subset \Pi_a$ such that for each $D \in U$, there exists a $G$-stable closed subscheme $T \subset S$ satisfying:
\begin{enumerate}[nolistsep]
    \item the restriction $\sheaf{E}|_{S \times D}$ is flat over $S \setminus T$,
    \item $D$ is smooth and $\sheaf{E}_s$-regular for all $s \in S \setminus T$,
    \item each $\sheaf{E}_{s_i}|_D$ is pure of dimension $d -1$ on $D$,
    \item for each $i$, $T \cap S_i$ has codimension $\geq 2$ in $S_i$.
\end{enumerate}
\end{lemma}
\begin{proof}
Since the property of being a pure sheaf is open, there exists a largest $G$-stable open subset $V \subset S$, containing all the $s_i$, such that $\sheaf{E}$ has pure fibers of dimension $d$ over $V$. By Lemma~\ref{lemma:flatRestriction}, for $a \gg 0$ there exists an open subset $U_0 \subset \Pi_a$ such that $\sheaf{E}|_{S \times D}$ is flat over $V$ for every $D \in U_0$. Moreover, by shrinking $U_0$ if necessary, we may also assume that each $D \in U_0$ is smooth and the restriction $\sheaf{E}_{s_i}|_D$ remains pure, cf.~\cite[Cor.~1.1.14 ii)]{huybrechts2010geometry}.

Next pick an index $1 \leq i \leq k$. If $S_i \cap V$ is empty, set $U_i \coloneqq \Pi_a$. Otherwise, choose a point $s'_i \in S_i \setminus V$. By \cite[Lem.~1.1.12]{huybrechts2010geometry} and Bertini's Theorem, there is an open subset $U_i \subset \Pi_a$ such that each $D \in U_i$ is smooth and $\sheaf{E}_{s'_i}$-regular. Therefore, by using \cite[Thm.~22.5]{matsumura1989commutative}, for each $D \in U_i$ we can find a $G$-stable open neighborhood $N(s'_i,D) \subset S$ of $s'_i$ over which $\sheaf{E}|_{S \times D}$ is flat. 

Consider the open subset $U \coloneqq U_0 \cap U_1 \cap \ldots \cap U_k$ of $\Pi_a$. Then, for each $D \in U$, the set
\begin{align*}
    T = S \setminus (V \cup \bigcup_{S_i \cap V \neq \emptyset} N(s'_i,D))
\end{align*}
satisfies the desired properties.
\end{proof}

\begin{cor}\label{lemma:completeRestriction}
Let $\sheaf{E}$ be a $G$-equivariant $S$-flat family of pure sheaves of dimension $d > 0$ on $X$ and let $S_1,\ldots,S_k$ be $G$-stable irreducible closed subschemes of $S$. For $a \gg 0$ and a general smooth complete intersection $X^{(d-1)} \subset X$ of $d-1$ divisors $D_1,\ldots,D_{d-1}$ in $\Pi_a$, there exists a $G$-stable closed subscheme $T \subset S$ satisfying:
\begin{enumerate}[nolistsep]
    \item the restriction $\sheaf{E}|_{S \times X^{(d-1)}}$ is flat over $S \setminus T$,
    \item the sequence $\underline{D} = D_1,\ldots,D_{d-1}$ is $\sheaf{E}_s$-regular for all $s \in S \setminus T$,
    \item for each $i$, $T \cap S_i$ has codimension $\geq 2$ in $S_i$.
\end{enumerate}
\end{cor}
\begin{proof}
Apply inductively Lemma~\ref{lemma:nonflatRestriction}.
\end{proof}

\begin{remark}\label{remark:avoidingPoints}
\begin{enumerate}[wide, labelwidth=!,nosep]
  \item For any point $s \in S$, we can find a subscheme $T \subset S$ in Lemma~\ref{lemma:nonflatRestriction}, resp.~Corollary~\ref{lemma:completeRestriction}, that avoids the $G$-orbit of $s$. For this, it is enough to choose the closure of the $G$-orbit of $s$ among the $S_i$.
  \item Since $G$ is connected, the irreducible components of $S$ may be chosen among the $S_i$ in Lemma~\ref{lemma:nonflatRestriction}.
  \item If $\sheaf{E}$ is a $G$-equivariant $S$-flat family of pure sheaves of dimension $2$, then one may take $T = \emptyset$ in Corollary~\ref{lemma:completeRestriction}. Indeed, this is a special case of Lemma~\ref{lemma:flatRestriction}.
\end{enumerate}
\end{remark}

\section{Semiampleness of determinant line bundles}\label{section:LineBundles}

In this section, we introduce a class of determinant line bundles which correspond functorially to any $S$-flat family of slope-semistable sheaves and then show their semiampleness. We will work under the following set-up:

\begin{setup}
Fix a class $c \in \grothnum{X}$ of dimension $d > 0$ and multiplicity $r$. Let $G$ be a connected algebraic group over $\C$ acting on a $\C$-scheme $S$ of finite type and $\sheaf{E}$ a $G$-equivariant $S$-flat family of sheaves of class $c$ on $X$.
\end{setup}

Consider the map $\lambda_\sheaf{E} : K(X) \to \Pic^G(S)$ given by the composition:
    \begin{align*}
        K(X) \xrightarrow{p_2^*} K^0(S \times X) \xrightarrow{\cdot [\sheaf{E}]} K^0(S \times X) \xrightarrow{{p_1}_{!}} K^0(S) \xrightarrow{\det} \Pic^G(S).
    \end{align*}
Note that if $u \in K(X)$ is a class satisfying $\chi(c \cdot u) = 0$, then
\begin{align*}
    \lambda_{\sheaf{E} \otimes p_1^*(L)}(u) \cong \lambda_{\sheaf{E}}(u)^{\rk(L)} \otimes L^{\chi(c \cdot u)} \cong \lambda_{\sheaf{E}}(u).
\end{align*}
for any line bundle $L$ on $S$. Also, one can easily see that $\lambda_\sheaf{E}$ behaves well under base-change. We refer the reader to \cite{le1996module} or \cite{le1992fibre} for a detailed account of the properties of $\lambda_\sheaf{E}$.

\subsection{Grothendieck classes of sheaves}

Pick an index $1 \leq j \leq d-1$ and choose $d-j$ divisors $D_1,\ldots,D_{d-j} \in \Pi_a$ such that $X^{(d-j)} \coloneqq D_1 \cap \cdots \cap D_{d-j}$ is a smooth complete intersection in $X$. For $l' > m' > 0$ consider the following class in $K(X)$:
\begin{align*}
    w_{l',m'} \coloneqq \chi(c(m') \cdot h^{d-j})[\sheaf{O}_X(l')] - \chi(c(l') \cdot h^{d-j})[\sheaf{O}_X(m')],
\end{align*}
where $h$ denotes the Grothendieck class $[\sheaf{O}_H] \in K(X)$ of some divisor $H \in \Pi_a$. The restriction of $w_{l',m'}$ to $X^{(d-j)}$ gives the class
\begin{align*}
    w_{l',m'}|_{X^{(d-j)}} = \chi(c|_{X^{(d-j)}}(m'))[\sheaf{O}_{X^{(d-j)}}(l')] - \chi(c|_{X^{(d-j)}}(l'))[\sheaf{O}_{X^{(d-j)}}(m')]
\end{align*}
in $K(X^{(d-j)})$.

\begin{remark}
\begin{enumerate}[wide, labelwidth=!,nosep]
\item If $n=d=2$, then we recover the classes defined in the torsion-free case (up to a multiplication by $l'-m'$) by Huybrechts--Lehn \cite[p.~223]{huybrechts2010geometry}.

\item Clearly $\chi(c \cdot w_{l',m'} \cdot h^{d-j}) = 0$, so one can use these classes to form determinant line bundles on the Simpson moduli space $M^{Gss}(c)$ of Gieseker-semistable sheaves of class on $X$ (see \cite[Thm.~8.1.5]{huybrechts2010geometry}).
\end{enumerate}
\end{remark}

One justification for the choice of these Grothendieck classes is provided by Lemma~\ref{lemma:lePotier} below. Before we state the result, some preparation is in place.

Denote by $P'$ the Hilbert polynomial of $c|_{X^{(d-j)}}$ and
\begin{align*}
    V' \coloneqq \C^{\oplus P'(m')}, \quad  \sheaf{H}' \coloneqq V' \otimes \sheaf{O}_{X^{(d-j)}}(-m').
\end{align*}
Consider the Quot scheme $\Quot_{X^{(d-j)}} \coloneqq \Quot_{X^{(d-j)}}(\sheaf{H}',P')$, endowed with the natural action of $G' \coloneqq \SL{V'}$. For $l' \gg m'$, Grothendieck \cite{grothendieck1961techniques} showed that there is a ($G'$-equivariant) closed immersion
\begin{align*}
    \varphi_{l'} : \Quot_{X^{(d-j)}} \to \Grass_{l'} \coloneqq \Grass(H^0(\sheaf{H}'(l')),P'(l'))
\end{align*}
that associates to a quotient $[q: \sheaf{H}' \to F]$ the corresponding quotient space $H^0(\sheaf{H}'(l')) \to H^0(F(l'))$. Denote by $H_{l'}$ the pullback via $\varphi_{l'}$ of the canonical line bundle $\sheaf{O}_{\Grass_{l'}}(1)$ on $\Grass_{l'}$ given by the Plücker embedding. Let $Q \subset \Quot_{X^{(d-j)}}$ be the open subscheme consisting of quotients $[q:\sheaf{H}' \to F]$ such that 
\begin{enumerate}[nosep]
 \item $H^i(F(m'))=0$ for all $i > 0$,
 \item the induced map $V' \to H^0(F(m'))$ is an isomorphism.
\end{enumerate}
Clearly $Q$ is invariant under the $G'$-action. This induces also a $G'$-linearization of $H_{l'}$ over $Q$, cf. \cite[p.~101]{huybrechts2010geometry}. 

We are now ready to prove the announced lemma, which is a generalization of \cite[Prop.~1.3]{le1996module} and relates the $G'$-equivariant line bundles $H_{l'}$ and $\lambda_{\sheaf{F}'}(w_{l,m}|_{X^{(d-j)}})$ over $Q$. Here and in what follows, $\sheaf{F}'$ denotes the universal quotient over $\Quot_{X^{(d-j)}}$.

\begin{lemma}\label{lemma:lePotier}
We have 
\begin{align*}
    \chi(c|_{X^{(d-j)}}(m')) H_{l'} = \lambda_{\sheaf{F}'}(w_{l',m'}|_{X^{(d-j)}})
\end{align*}
in $\Pic^{G'}(Q)$ with additive notation.
\end{lemma}

\begin{proof}
By \cite[Prop.~2.2.5]{huybrechts2010geometry}, $H_{l'}$ can be rewritten as
\begin{align*}
    H_{l'}  = \lambda_{\sheaf{F}'}([\sheaf{O}_{X^{(d-j)}}(l')]).
\end{align*}
On the other side, we have ${R^ip_1}_*(\sheaf{F}'(m')) = 0$ for all $i > 0$ and ${p_1}_*(\sheaf{F}'(m')) \cong V' \otimes \sheaf{O}_Q$ over $Q$. Therefore,
\begin{align*}
    \lambda_{\sheaf{F}'}([\sheaf{O}_{X^{(d-j)}}(m')]) \cong \det({p_1}_*(\sheaf{F}'(m'))) \cong \det(V')\otimes \sheaf{O}_Q
\end{align*}
is trivial in $\Pic^{G'}(Q)$. We obtain
\begin{align*}
    \lambda_{\sheaf{F}'}(w_{l',m'}|_{X^{(d-j)}}) &= \chi(c|_{X^{(d-j)}}(m')) \lambda_{\sheaf{F}'}([\sheaf{O}_{X^{(d-j)}}(l')]) - \chi(c|_{X^{(d-j)}}(l')) \lambda_{\sheaf{F}'}([\sheaf{O}_{X^{(d-j)}}(m')]) \\
      &= \chi(c|_{X^{(d-j)}}(m'))H_{l'}
\end{align*}
in $\Pic^{G'}(Q)$ (with additive notation).
\end{proof}

The lemma below will be used implicitly in the next section to assume that the morphism $\varphi_{l'}$ embeds the Quot scheme into the Grassmannian for any complete intersection $X^{(d-j)}$. 

\begin{lemma}\label{lemma:quotEmbedding}
For $l' \gg 0$, the morphism $\varphi_{l'}$ is a closed immersion for every smooth complete intersection $X^{(d-j)} \subset X$ of $d-j$ divisors in $\Pi_a$.
\end{lemma}
\begin{proof}
Any such complete intersection $X^{(d-j)}$ may be seen as a closed point in the Hilbert scheme $\Hilb_X(P)$, where $P$ is the Hilbert polynomial of $X^{(d-j)}$. Indeed, $P$ does not depend on $X^{(d-j)}$, as one easily checks by using the Koszul resolution of $X^{(d-j)}$ to compute its Hilbert polynomial. Denote by $Z \subset \Hilb_X(P) \times X$ the universal closed subscheme. Then $\Quot_{X^{(d-j)}}$ is a closed point of the relative Quot scheme
\begin{align*}
   \Quot_{Z}(V' \otimes \sheaf{O}_{Z}(-m'),P') \to \Hilb_X(P).
\end{align*}
Since this construction is functorial, it is enough to choose $l'$ sufficiently large such that
\begin{align*}
    \varphi : \Quot_{Z}(V' \otimes \sheaf{O}_{Z}(-m'),P') \to \Grass(V' \otimes H^0(Z,\sheaf{O}_{Z}(l'-m')),P'(l')),
\end{align*}
is the Grothendieck embedding. 
\end{proof}

Next assume that the sequence $\underline{D} = D_1, \ldots, D_{d-j}$ is $\sheaf{E}_s$-regular for every $s \in S$. Set $\sheaf{E}_i \coloneqq \sheaf{E} \otimes p_2^*(\Lambda^i \C^{d-1}\otimes \sheaf{O}_X(-ia))$ and consider the Koszul complex
\begin{align*}
    K_\bullet(\sheaf{E},\underline{D}) :\quad  0 \to \sheaf{E}_{d-j} \to \sheaf{E}_{d-2} \to \ldots \to \sheaf{E}_0
\end{align*}
corresponding to $\underline{D}$. Note that all the $\sheaf{E}_i$ are $S$-flat and the restriction of $K_\bullet(\sheaf{E},\underline{D})$ to each fiber over $S$ is exact, since $\underline{D}$ is $\sheaf{E}_s$-regular by assumption. According to \cite[\href{https://stacks.math.columbia.edu/tag/00MI}{Tag 00MI}]{stacks-project}, the complex $K_\bullet(\sheaf{E},\underline{D})$ is exact and $\sheaf{E}|_{S \times X^{(d-j)}} = \Coker(\sheaf{E}_1 \to \sheaf{E}_0)$ is $S$-flat.

\begin{lemma}\label{lemma:Giso}
There is a $G$-equivariant isomorphism
\begin{align*}
    \lambda_\sheaf{E}(w_{l',m'} \cdot h^{d-j}) \cong \lambda_{\sheaf{E}|_{S \times X^{(d-j)}}}(w_{l',m'}|_{X^{(d-j)}})
\end{align*}
of determinant line bundles over $S$.
\end{lemma}
\begin{proof}
By applying the map $\lambda_{(-)}(w_{l',m'})$ to the Koszul resolution $K_\bullet(\sheaf{E},\underline{D})$ and using \cite[Lem.~8.1.2 i)]{huybrechts2010geometry}, we get the following chain of $G$-equivariant isomorphisms:
\begin{align*}
    \lambda_{\sheaf{E}|_{S \times X^{(d-j)}}}(w_{l',m'}|_{X^{(d-j)}})
    &\cong \bigotimes_{i=0}^{d-j} \lambda_{\sheaf{E}_i}(w_{l',m'})^{(-1)^i}\\
    &\cong \bigotimes_{i=0}^{d-j} \lambda_{\sheaf{E}}(w_{l',m'} \cdot [\Lambda^i E \otimes \sheaf{O}_X(-ia)])^{(-1)^i}\\
    &\cong \lambda_{\sheaf{E}}(\sum_{i=0}^{d-j} (-1)^i w_{l',m'} \cdot [\Lambda^i E \otimes \sheaf{O}_X(-ia)])\\
    &\cong \lambda_{\sheaf{E}}(w_{l',m'}\cdot [\sheaf{O}_{X^{(d-j)}}]).
\end{align*}
\end{proof}

\subsection{Semiampleness theorems for line bundles}\label{subsect:semiamplenessTheorems}

We continue in the notation of the previous section but take $j = 1$. So in what follows,
\begin{align*}
    w_{l',m'} \coloneqq \chi(c(m') \cdot h^{d-1})[\sheaf{O}_X(l')] - \chi(c(l') \cdot h^{d-1})[\sheaf{O}_X(m')].
\end{align*}
Consider the $G$-linearized determinant line bundle
\begin{align*}
  \sheaf{L}_{l',m'} \coloneqq \lambda_\sheaf{E}(w_{l',m'} \cdot h^{d-1})
\end{align*}
over $S$. Under the assumption that $S$ is \textit{weakly normal}, we next prove that $\sheaf{L}_{l',m'}$ is $G$-semiample for $l' \gg m' \gg 0$ and $a \gg 0$, i.e.~there is an integer $\nu > 0$ such that $\sheaf{L}_{l',m'}^\nu$ is globally generated by $G$-invariant sections. We refer the reader to \cite[Section~2.3]{greb2017compact} for the definition of weak normality and some of its useful properties.

We assume that $l' > 0$ is large enough such that the conclusion of Lemma~\ref{lemma:quotEmbedding} holds true. That is, for any smooth complete intersection $X^{(d-1)} \subset X$ of $d-1$ divisors in $\Pi_a$, the scheme $\Quot_{X^{(d-1)}}$ embeds into the Grassmannian via the Grothendieck morphism $\varphi_{l'} : \Quot_{X^{(d-1)}} \to \Grass_{l'}$,  as described in the previous section.

Let $s_0 \in S$ be a closed point, $S_1$ its $G$-orbit in $S$ and $S_2,\ldots,S_k$ the irreducible components of $S$. By Corollary~\ref{lemma:completeRestriction}, for a general smooth complete intersection $X^{(d-1)}$, there exists a $G$-stable closed subvariety $T \subset S$ such that each $T \cap S_i$ has codimension $\geq 2$ in $S_i$, and $\sheaf{G}\coloneqq \sheaf{E}|_{(S\setminus T) \times X^{(d-1)}}$ is an $(S\setminus T)$-flat family of sheaves of dimension one on $X^{(d-1)}$. Moreover, $\underline{D}$ is $\sheaf{E}_s$-regular for all $s \in S\setminus T$,  and so the Koszul complex $K_\bullet(\sheaf{E},\underline{D})$ is exact over $S \setminus T$. For $m' \gg 0$, we may assume that each fiber $\sheaf{G}_s$ is $m'$-regular for $s \in S \setminus T$, cf.~\cite[Exam.~1.8.7]{lazarsfeld2017positivity}. In particular, ${p_1}_*(\sheaf{G}(m'))$ is a locally free $G$-equivariant $\sheaf{O}_{S \setminus T}$-sheaf of rank $P'(m')$.

Denote by $\mathcal{S}$ the projective frame bundle associated to ${p_1}_*(\sheaf{G}(m'))$ and let $\pi:\mathcal{S} \to S \setminus T$ be the canonical projection. Notice that the action of $G'$ on ${p_1}_*(\sheaf{G}(m'))$ induces a natural $G'$-action on $\mathcal{S}$. Also, $\mathcal{S}$ inherits a $G$-action from $S$ that is compatible with its $G'$-action. Then there exists a quotient
\begin{align*}
  \sheaf{O}_{\mathcal{S}} \otimes \sheaf{H'} \to \pi^*\sheaf{G} \otimes \sheaf{O}_\pi(1),
\end{align*}
which induces a $G \times H$-equivariant morphism
\begin{align*}
  \Phi : \mathcal{S} \to \Quot_{X^{(d-1)}}.
\end{align*}
Above $\Quot_{X^{(d-1)}}$ is endowed with the trivial $G$-action. Thus we have a diagram as follows
\begin{center}
\begin{tikzcd}[column sep=normal]
\mathcal{S} \ar[d,"\pi"] \ar[r,"\Phi_{\sheaf{G}}"] &  \Quot_{X^{(d-1)}} \ar[d,dashed,"\varphi'"]  \\ 
 S\setminus T \ar[r,dashed] & M_{X^{(d-1)}}
\end{tikzcd}
\end{center}
where $M_{X^{(d-1)}}$ denotes the Simpson moduli space of slope-semistable sheaves of Hilbert polynomial $P'$ on ${X^{(d-1)}}$.

\begin{remark}
\begin{enumerate}[wide, labelwidth=!,nosep]
\item If $\sheaf{F}'$ denotes the universal family of quotients over $\Quot_{X^{(d-1)}}$, then $H_{l'} \coloneqq \lambda_{\sheaf{F}'}([\sheaf{O}_{X^{(d-1)}}(l')])$ is a very ample line bundle over $\Quot_{X^{(d-1)}}$ for $l' \gg m'$, cf. \cite[Prop.~2.2.5]{huybrechts2010geometry}. Moreover, $\Quot_{X^{(d-1)}}$ is endowed with a natural $G' \coloneqq \SL{V'}$ action, which further induces a $G'$-linearization of $H_{l'}$.  Now consider the $G$-stable open subscheme $R_{X^{(d-1)}} \subset \Quot_{X^{(d-1)}}$ consisting of $m'$-regular Gieseker-semistable quotients $[q: V' \otimes \sheaf{O}_{X^{(d-1)}}(-m') \to F]$ such that the induced map $V' \to H^0(F(m'))$ is an isomorphism. By Simpson's construction \cite{simpson1994moduli}, for $l' \gg m' \gg 0$, $R_{X^{(d-1)}}$ is the locus of GIT-semistable points of the closure $\overline{R_{X^{(d-1)}}} \subset \Quot_{X^{(d-1)}}$ with respect to $H_{l'}$. Furthermore, the rational map $\varphi'$ is well-defined over $R_{X^{(d-1)}}$ such that $R_{X^{(d-1)}} \to M_{X^{(d-1)}}$ is a good GIT quotient, in the sense of Mumford \cite{mumfordGIT}. 

\item It is important to note that the image of $\Phi_{\sheaf{G}}$ is actually contained in the open subscheme $Q \subset \Quot_{X^{(d-1)}}$ defined in the previous section. That is, $Q$ is the open subscheme consisting of quotients $[q:\sheaf{H}' \to G]$ such that 
\begin{enumerate}[(i),nosep]
 \item $H^i(F(m'))=0$ for all $i > 0$,
 \item the induced map $V' \to H^0(F(m'))$ is an isomorphism.
\end{enumerate}
Note that $R_{X^{(d-1)}}$ is an open subset of $Q$.
\end{enumerate}
\end{remark}

We have the following chain of isomorphisms:
\begin{align*}
    \Phi_{\sheaf{G}}^*(\lambda_{\sheaf{F}'}(w_{l',m'}|_{X^{(d-1)}})) &\cong \lambda_{\pi^*\sheaf{G} \otimes \sheaf{O}_\mathcal{S}(1)}(w_{l',m'}|_{X^{(d-1)}}) &&\textnormal{by \cite[Lem.~8.1.2 ii)]{huybrechts2010geometry}}\\
     &\cong \lambda_{\pi^*\sheaf{G}}(w_{l',m'}|_{X^{(d-1)}}) &&\textnormal{by \cite[Lem.~8.1.2 iv)]{huybrechts2010geometry}}\\
     &\cong \pi^*\lambda_{\sheaf{G}}(w_{l',m'}|_{X^{(d-1)}}) &&\textnormal{by \cite[Lem.~8.1.2 ii)]{huybrechts2010geometry}}\\
     &\cong \pi^*\lambda_{\sheaf{E}}(w_{l',m'} \cdot h^{d-1}) &&\textnormal{by Lemma~\ref{lemma:Giso}}\\
     &\cong \pi^*\sheaf{L}_{l',m'} &&\text{by definition}.
\end{align*}
Clearly the above isomorphisms are compatible with the group actions. Let $\sigma$ be a $G'$-invariant section in $H^0(\Quot_{X^{(d-1)}}, H_{l'}^{\nu'})^{G'}$ for some $\nu' > 0$. By Lemma~\ref{lemma:lePotier}, its restriction to $Q$ yields a $G'$-invariant section $\sigma_Q$ in $H^0(Q, \lambda_{\sheaf{F}'}(w_{l',m'}|_{X^{(d-1)}})^{\nu})^{G'}$ for some $\nu > 0$. Then $\Phi_\sheaf{G}^*(\sigma_Q)$ is a $G \times G'$-invariant section which descends to a $G$-invariant section in $H^0(S \setminus T, \sheaf{L}_{l',m'}^\nu)^G$ since $\pi$ is a good quotient. As $S$ is weakly normal, we can extend this section to the whole $S$ by using a result due to Greb--Toma \cite[Lem.~2.12]{greb2017compact}, which we rewrite in our notation below.

\begin{lemma}\label{lemma:extension}
Under the above assumptions, there exists a finite system of irreducible subvarieties $(S'_j)_{j=1,\ldots,t}$ with the following property: For any $G$-stable closed subvariety $T$ of $S$ such that
\begin{enumerate}[nolistsep]
    \item the intersection of $T$ with each irreducible component of $S$ has codimension $\geq 2$ in $S$, and
    \item $T$ contains none of the $S'_j$,
\end{enumerate}
any $G$-invariant section of $H^0(S \setminus T, \sheaf{L}_{l',m'})^G$ extends to a $G$-invariant section of $H^0(S, \sheaf{L}_{l',m'})^G$.
\end{lemma}

Clearly, we may assume from the beginning that $T$ contains none of the $S'_j$ given by Lemma~\ref{lemma:extension}, and so we obtain a $G$-invariant section $\overline{\sigma} \in H^0(S,\sheaf{L}_{l',m'}^\nu)^G$. Therefore, we get a map
\begin{align*}
    \Gamma_{\sheaf{E}}: H^0(\Quot_{X^{(d-1)}}, H_{l'}^{\otimes \nu'})^{G'} \to H^0(S,\sheaf{L}_{l',m'}^{\otimes \nu})^G.
\end{align*}

We show below that inside $\Quot_{X^{(d-1)}}$ slope-semistability coincides with the notion of GIT-semistability (with respect to $H_{l'}$) introduced by Mumford \cite{mumfordGIT}.
\begin{lemma}\label{lemma:gitSs}
Let $[q: \sheaf{H'} \to F] \in \Quot_{X^{(d-1)}}$ be a quotient. Then $F$ is GIT-semistable if and only if it is slope-semistable.
\end{lemma}
\begin{proof}
By Simpson's construction, it is enough to check that a GIT-semistable quotient $[q: \sheaf{H'} \to F] \in \Quot_{X^{(d-1)}}$ is pure. Firstly, by \cite[Cor.~4.4.7]{huybrechts2010geometry}, we have that $V' \to H^0(F(m'))$ is injective. Let $F'(m')$ be the subsheaf of $F(m')$ defined by the sections in $V'$. By \cite[Lem.~4.4.6]{huybrechts2010geometry}, it follows that $P(F') \geq P(F)$. Thus $F' = F$ and in particular $V' \to H^0(F(m'))$ is a bijection.

Applying again \cite[Cor.~4.4.7]{huybrechts2010geometry} to a torsion subsheaf $T \subset F$, we get $H^0(T(m'))=0$. But since $T(m')$ has dimension $0$, it is completely determined by its global sections. In conclusion $T = 0$ and $F$ is pure.
\end{proof}

\begin{lemma}\label{lemma:invariantSections}
For $l' \gg m' \gg 0$ and $a \gg 0$, there is an integer $\nu > 0$ and a $G$-invariant section in $\sheaf{L}_{l',m'}^\nu$ that does not vanish at $s_0$.
\end{lemma}
\begin{proof}
By Theorem~\ref{thm:RestrictionSemi}, for $a \gg 0$ and a general smooth complete intersection $X^{(d-1)}$, the restriction $\sheaf{E}_{s_0}|_{X^{(d-1)}}$ remains $\hat{\mu}$-semistable, and thus also GIT-semistable by Lemma~\ref{lemma:gitSs}. Then there exists an integer $\nu' > 0$ and a $G'$-invariant section $\sigma \in H^0(\Quot_{X^{(d-1)}}, H_{l'}^{\otimes \nu'})^{G'}$ non-vanishing at the closed point corresponding to $\sheaf{E}_{s_0}|_{X^{(d-1)}}$ in $\Quot_{X^{(d-1)}}$. By construction, the $G$-invariant section $\Gamma_{\sheaf{E}}(\sigma)$ will not vanish at $s_0 \in S$.
\end{proof}

Putting together the above considerations, we conclude:

\begin{theorem}\label{thm:semiAmpleness}
Let $S$ be a weakly normal algebraic $G$-variety and $\sheaf{E}$ a $G$-equivariant $S$-flat family of $\hat{\mu}$-semistable sheaves of class $c$ on $X$. Then, for $l' \gg m' \gg 0$ and $a \gg 0$, there is an integer $\nu > 0$ such that the line bundle $\sheaf{L}_{l',m'}^\nu$ is generated by $G$-invariant global sections over $S$.
\end{theorem}
\begin{proof}

By Lemma~\ref{lemma:invariantSections}, for any closed point $s \in S$ and $a \gg 0$, there exists a large enough integer $\nu(s) > 0$ and a $G$-invariant section in $\sheaf{L}_{l',m'}^{\nu(s)}$ that does not vanish at $s$. Since $S$ is Noetherian, we may choose a large enough power $\nu$ that works for all $s \in S$.
\end{proof}

If $d = 2$, there is no need to assume the weak normality hypothesis on $S$. Indeed, as we have noticed in Remark~\ref{remark:avoidingPoints}, in this case we may choose $T = \emptyset$. We obtain the following result:

\begin{theorem}\label{thm:semiAmpleness2}
Let $S$ be an algebraic $G$-variety and $\sheaf{E}$ a $G$-equivariant $S$-flat family of $\hat{\mu}$-semistable $2$-dimensional sheaves of class $c$ on $X$. Then, for $l' \gg m' \gg 0$ and $a \gg 0$, there is an integer $\nu > 0$ such that the line bundle $\sheaf{L}_{l',m'}^\nu$ is generated by $G$-invariant global sections over $S$.
\end{theorem}
\section{$G$-properness, relative $G$-normalization and Iitaka $G$-fibration}\label{sect:Gproperness}

In this section, fix a connected algebraic group $G$ over $\C$. Before we proceed with the construction of the moduli space, we make a short detour to introduce the notion of $G$-properness and define what we call the relative $G$-normalization. We then construct the Iitaka $G$-fibration of a $G$-proper scheme endowed with a $G$-semiample line bundle. In a sense, we generalize the results described by Lazarsfeld in \cite[Ch.~2]{lazarsfeld2017positivity} to the $G$-equivariant setting.

In what follows, we work over the category $\schC$ of schemes of finite type over $\C$ and so every morphism will be of finite type.

\begin{definition}\label{def:Gproper}
Let $R$ be a $G$-scheme and $f : R \to S$ a $G$-invariant morphism. We say that $f$ is $G$-\textit{universally closed} if for every commutative diagram
\begin{center}
\begin{tikzcd}[column sep=normal]
\Spec(K) \ar[d,"j"] \ar[r,"g"] &  R \ar[d,"f"] \\  \Spec(A) \ar[ru,dashed,"i"] \ar[r,"h"] & S
\end{tikzcd}
\end{center}
where $A$ is a discrete valuation ring over $\C$ of quotient field $K$, there exists a morphism $i : \Spec(A) \to R$ such that $i \circ j$ and $g$ differ by a group element in $G(K)$. If in addition $f$ is also separated, we call $f$ a $G$-\textit{proper} morphism.
\end{definition}

\begin{remark}\label{remark:Gproperties}
When $G$ acts trivially we recover the classical notion of properness on schemes. As in the classical setting, one can easily check the following properties of $G$-properness:
\begin{enumerate}[nosep]
  \item $G$-proper morphisms are stable under base changes $S' \to S$.
  \item If $f : R \to Z$ and $g : Z \to S$ are two morphisms such that $g  \circ f$ is $G$-proper, $f$ is $G$-invariant and $g$ is separated, then $f$ is $G$-proper.
  \item If $f : R \to S$ is $G$-proper and $g : S \to T$ is proper, then $g \circ f$ is also $G$-proper.
\end{enumerate}
Similar properties hold also for $G$-universally closed morphisms.
\end{remark}

\begin{lemma}\label{lemma:closedImage}
Given a commutative diagram
\begin{center}
\begin{tikzcd}[column sep=small]
R \arrow[rr,"f"] \arrow[rd,"p"] & & S \arrow[dl,"q"']  \\
    & T &
\end{tikzcd}
\end{center}
such that $p$ is $G$-universally closed, $f$ is $G$-invariant and $q$ is separated, then the image of $f$ is proper over $T$.
\end{lemma}
\begin{proof}
By Nagata's Compactification Theorem, we may assume that $S$ is a proper scheme over $T$ and reduce the problem to showing that $\Im(f)$ is a closed subscheme of $S$. Since $f$ is of finite type, it is enough to show that $\Im(f)$ is stable under specialization, cf. \cite[Ch.~II, Lem.~4.5]{hartshorne1977algebraic}. So let $x_1 \in R$ be a point and consider a specialization $s_0 \in S$ of its image $s_1 = f(x_1)$. Denote by $\sheaf{O}$ the local ring of $s_0$ on $\overline{\{s_1\}}$ with its reduced induced structure. Then the quotient field of $\sheaf{O}$ is $k(s_1)$, and we have a finitely generated field extension $k(s_1) \subset k(x_1)$. Set $K \coloneqq k(x_1)$ and consider a discrete valuation ring $A$ of $K$ dominating $\sheaf{O}$, which exists cf. \cite[Ch.~II, Ex.~4.11]{hartshorne1977algebraic}. This induces a commutative diagram
\begin{center}
\begin{tikzcd}
\Spec(K) \arrow[r,"g"] \arrow[d,"j"]
& R \arrow[d,"f"] \\
\Spec(A) \arrow[r,"h"]
& S
\end{tikzcd}
\end{center}
such that $h: \Spec(A) \to S$ sends the generic point of $A$ to $s_1$ and its closed point to $s_0$. By Remark~\ref{remark:Gproperties}, $f$ is a $G$-universally closed morphism. Thus there exists a morphism $i : \Spec(A) \to R$ such that $i \circ j$ and $g$ differ by a group element in $G(K)$. As $f$ is $G$-invariant, it follows that $(q \circ f \circ i)|_{\Spec(K)}= (q \circ h)|_{\Spec(K)}$. But $q$ is separated, so $f \circ i = h$, which shows that $s_0$ is in the image of $f$.
\end{proof}

Next we introduce the notion of relative $G$-normalization, which is an equivariant version of the so-called relative normalization (see \cite[Tag~035H]{stacks-project}). 

\begin{definition}\label{Def:Normalization}
Let $R$ be a $G$-scheme and $f: R \to S$ a $G$-invariant morphism. If $\sheaf{O}'$ is the integral closure of $\sheaf{O}_S$ in $(f_* \sheaf{O}_R)^G$, then we call the $S$-scheme
\begin{align*}
\pi : S' = \underline{\Spec}_S\!\left(\sheaf{O}'\right) \to S
\end{align*}
the $G$-\textit{normalization} of $S$ in $R$. It comes equipped with a natural factorization of $f$ given by
\begin{align*}
R \xrightarrow{f'} S' \xrightarrow{\pi} S,
\end{align*}
where $f'$ is the composition of the canonical morphism $R \to \underline{\Spec}_S\!\left((f_*\sheaf{O}_R)^G \right)$ and the morphism of relative spectra coming from the inclusion map
$\sheaf{O}' \to (f_*\sheaf{O}_R)^G$. Note that in this case $f'$ is a $G$-invariant morphism and $\pi$ is integral by definition.
\end{definition}

As expected, the just defined normalization satisfies a universal property as stated below.

\begin{prop}
Under the same notation as in Definition~\ref{Def:Normalization}, the $G$-normalization $\pi: S' \to S$ satisfies the following universal property: for any other factorization $R \xrightarrow{g} Z \xrightarrow{\nu} S$ with $g$ a $G$-invariant morphism and $\nu$ integral, there exists a unique morphism $h: S' \to  Z$ making the diagram
\begin{center}
\begin{tikzcd}[column sep=normal]
R \ar[d,"f'"] \ar[r,"g"] &  Z \ar[d,"\nu"] \\  S' \ar[ru,"h"] \ar[r,"\pi"] & S
\end{tikzcd}
\end{center}
commute.
\end{prop}
\begin{proof}
The proof goes as in the non-equivariant case, see \cite[Tag 035I]{stacks-project}.
\end{proof}

The following is an equivariant version of the Stein factorization (see \cite[Tag~03H0]{stacks-project}). 

\begin{prop}\label{Prop:normalization}
Assume that $G$ is reductive and let $f : R \to S$ be a $G$-proper surjective morphism. Then the $G$-normalization of $S$ in $R$ exists and gives a factorization
\begin{center}
\begin{tikzcd}[column sep=normal]
R \arrow{r}{f'}  \arrow{rd}{f} 
  & S' \arrow{d}{\pi} \\
    & S
\end{tikzcd}
\end{center}
with the following properties:
\begin{enumerate}[(i),nosep]
\item $(f '_* \sheaf{O}_{R})^{G} = \sheaf{O}_{S'}$ and $S' = \underline{\Spec}_S\!\left((f_* \sheaf{O}_{R})^{G}\right)$,
\item the morphism $\pi : S' \to S$ is finite,
\item the morphism $f'$ is a surjective $G$-invariant morphism.
\end{enumerate}
\end{prop}
\begin{proof}
Clearly, the $G$-normalization $S'$ of $S$ in $R$ is defined and gives a natural factorization as in the statement. 
To prove $(1)$ and $(2)$, it is enough to show that $(f_* \sheaf{O}_{R})^{G}$ is finite over $\sheaf{O}_S$. We argue locally over an open affine subscheme $U = \Spec(A) \subset S$. Let $h \in \Gamma(f^{-1}(U),\sheaf{O}_{R})^{G}$ be a $G$-invariant section, which corresponds to a $G$-invariant morphism ${h: f^{-1}(U) \to \Spec(A[T])}$ making the diagram
\begin{center}
\begin{tikzcd}[column sep=normal]
f^{-1}(U) \arrow{r}{h}  \arrow{rd}{f} 
  & \Spec(A[T]) \arrow{d} \\
    & U
\end{tikzcd}
\end{center}
commute. By the properties of $G$-properness (see Remark~\ref{remark:Gproperties}), one gets that $f|_{f^{-1}(U)}$ is also $G$-proper. Therefore, if $Z$ denotes the image of $h$, then $Z$ is a proper scheme over $U$ by Lemma~\ref{lemma:closedImage}. Also, as a closed subscheme of $\Spec(A[T])$, $Z = \Spec(A[T]/I)$ for some ideal $I \subset A[T]$. We thus obtain that $Z \to U$ is affine and proper, and so $Z$ is integral over $U$. This means there exists a monic polynomial $P \in I$ such that $P(h)=0$, which proves that $A \to \Gamma(f^{-1}(U),\sheaf{O}_{\Rss})^{G}$ is integral. It remains to show that this morphism is also of finite type. The action of $G$ on $f^{-1}(U)$ induces a dual action of $\Gamma(G,\sheaf{O}_G)$ on $\Gamma(f^{-1}(U),\sheaf{O}_{R})$. But as $G$ is reductive, by Nagata's Theorem (see Theorem~\ref{thm:Nagata} below) it follows that $\Gamma(f^{-1}(U),\sheaf{O}_{R})^G$ is a finitely generated $\C$-algebra. In particular, $\Gamma(f^{-1}(U),\sheaf{O}_{R})^G$ is of finite type over $A$.

We show that $f'$ is surjective. By Lemma~\ref{lemma:closedImage} the image of $f'$ is a closed subscheme of $S'$, and thus it is determined by a coherent sheaf of ideals $\sheaf{I} \subset \sheaf{O}_{S'}$ such that $f'^* \sheaf{I} \to \sheaf{O}_{R}$ is zero. Then the composition $\sheaf{I} \to \sheaf{O}_{S'} \to (f'_* \sheaf{O}_{R})^{G}$ is also zero by adjointness. But $(f'_* \sheaf{O}_{R})^{G} = \sheaf{O}_{S'}$ by (1), which implies that the closed subscheme determined by $\sheaf{I}$ is exactly $S'$.
\end{proof}

\begin{theorem}[Nagata's Theorem]\label{thm:Nagata}
    If $A$ is a finitely generated $\C$-algebra endowed with the action of a reductive group $G$, then its ring of invariants $A^G$ is finitely generated.
\end{theorem}

\subsection{A devissage lemma for equivariant sheaves}

We continue in the notation of the previous section, hence $G$ is a connected algebraic group and let $R$ be a $G$-universally closed scheme over $\C$ (see Definition~\ref{def:Gproper}). We will prove a devissage lemma for $G$-equivariant sheaves on $R$. For this, we work within the category $\Coh^G(R)$ of $G$-equivariant coherent sheaves on $R$. We say an object $E$ of $\Coh^G(R)$ has property $\mathcal{P}$ if the vector space of $G$-invariant global sections $\Gamma(R,E)^G$ is finitely generated. 

We want to prove that any $G$-equivariant sheaf has $\mathcal{P}$, for which we need the following two standard lemmata.

\begin{lemma}\label{lemma:exactSequenceDevissage}
For any short exact sequence of $G$-equivariant sheaves on $R$
\begin{align*}
    0 \to E' \to E \to E'' \to 0,
\end{align*}
if $E'$ and $E''$ (resp. $E$ and $E''$) have $\mathcal{P}$, then the third has $\mathcal{P}$.
\end{lemma}
\begin{proof}
The proof is completely analogue to the one in the non-equivariant case, we omit it. 

\end{proof}

\begin{lemma}\label{lemma:globalSections}
If $Y$ is an integral $G$-stable closed subscheme of $R$, then $\Gamma(Y,\sheaf{O}_Y)^G = \C$.
\end{lemma}
\begin{proof}
   Let $s \in \Gamma(Y,\sheaf{O}_Y)^G$ be a $G$-invariant global section, which induces a $G$-invariant morphism $s: Y \to \A^1$. As $R$ is $G$-universally closed over $\C$, clearly so is $Y$. By Lemma~\ref{lemma:closedImage}, the image of $s$ is proper over $\C$, and so it consists of a finite number of points in $\A^1$. As $Y$ is integral, the image reduces to a single point.
\end{proof} 

\begin{theorem}\label{thm:devissageThm}
    Any $G$-equivariant sheaf on $R$ satisfies property $\mathcal{P}$.
\end{theorem}
\begin{proof}
We argue by noetherian induction as follows. Let $Y$ be a $G$-stable closed subscheme of $R$ and assume that any $G$-equivariant sheaf whose support is a proper subscheme $Z \subset Y$ has $\mathcal{P}$. We want to show that any $G$-equivariant sheaf supported on $Y$ has $\mathcal{P}$. 

Let $E$ be a $G$-equivariant sheaf on $Y$. First, we reduce to the case when $Y$ is reduced. In case $Y$ is not reduced, let $\sheaf{I}$ be the sheaf ideal corresponding to the reduced induced subscheme structure on $Y$. Note that we can naturally endow $\sheaf{I}$ with the structure of a $G$-equivariant sheaf. Also there exists an integer $n > 0$ such that $\sheaf{I}^n = 0$. We have a filtration of $G$-equivariant sheaves
\begin{align*}
    0 = \sheaf{I}^nE \subset \sheaf{I}^{n-1}E \subset \cdots \subset  \sheaf{I}E \subset E
\end{align*}
such that the successive quotients $\sheaf{I}^{j}E/\sheaf{I}^{j+1}E$ are annihilated by $\sheaf{I}$. Hence these quotients are supported on $Y_\text{red}$, and so they satisfy $\mathcal{P}$ by assumption. By splitting the filtration above in short exact sequences and using Lemma~\ref{lemma:exactSequenceDevissage}, we get that $E$ has $\mathcal{P}$.

Now we assume that $Y$ is reduced. In case $Y$ is not irreducible, let $Y_1, \ldots, Y_r$ be its irreducible components. Note that the $Y_i$ are $G$-stable since $G$ is connected. Let $\sheaf{I}_1$ be the ideal sheaf of $Y_1$ with its natural $G$-equivariant structure and consider the exact sequence
\begin{align*}
    0 \to \sheaf{I}_1 E \to E \to E/\sheaf{I}_1 E \to 0.
\end{align*}
Since $\sheaf{I}_1 E$ is supported on $Y_2 \cup \cdots \cup Y_r$ and $E/\sheaf{I}_1 E$ is supported on $Y_1$, we may assume by noetherian induction that both sheaves satisfy $\mathcal{P}$. By Lemma~\ref{lemma:exactSequenceDevissage}, we deduce that $E$ has also $\mathcal{P}$.

Finally, we assume that $Y$ is integral and show that $E$ has $\mathcal{P}$. We argue by induction on the rank of $E$. If $\rank(E) =0$, then $E$ is supported on a proper subscheme of $Y$ and so $E$ has $\mathcal{P}$. Now suppose that $\rank(E) > 0$. If we denote by $F$ the maximal torsion subsheaf of $E$, then we obtain a short exact sequence in $\Coh^G(R)$ given by
\begin{align*}
    0 \to F \to E \to E/F \to 0.
\end{align*}
Note that $F$ has $\mathcal{P}$ as it is supported on a proper subscheme. Using again Lemma~\ref{lemma:exactSequenceDevissage} for the above sequence, we reduce to the case $E$ is pure. We may suppose that $E$ has at least a $G$-invariant global section, say $s$.
We have a short exact sequence
\begin{align*}
    0 \to \sheaf{O}_Y \xrightarrow{\cdot s} E \to Q \to 0
\end{align*}
in $\Coh^G(R)$ such that $\rank(Q) < \rank(E)$. Hence $Q$ has $\mathcal{P}$ by induction on rank. As $\sheaf{O}_Y$ has also $\mathcal{P}$ by Lemma~\ref{lemma:globalSections}, we conclude that $E$ has $\mathcal{P}$. 
\end{proof}

\subsection{Iitaka $G$-fibration}\label{subsect:Iitaka}

Here we assume that $G$ is reductive. Let $R$ be a $G$-proper scheme over $\C$ and $\sheaf{L}$ a $G$-equivariant $G$-semiample line bundle on $R$, i.e. there exists $\nu > 0$ such that $\sheaf{L}^\nu$ is globally generated by $G$-invariant sections. Then the set
\begin{align*}
    M(R,\sheaf{L}) = \{ m \in \N \mid \sheaf{L}^{\otimes m} \text{ is generated by $G$-invariant sections} \}
\end{align*}
is nonempty. As $R$ is $G$-proper, for each $\nu \in M(R,\sheaf{L})$ the vector space $\Gamma(R,\sheaf{L}^\nu)^G$ is finitely generated by Theorem~\ref{thm:devissageThm} and induces a $G$-invariant morphism
\begin{align*}
\varphi_\nu : R \to \CP(\Gamma(R,\sheaf{L}^\nu)^G)
\end{align*}
such that ${\varphi}_\nu^*\sheaf{O}(1) \cong \sheaf{L}^\nu$. According to Lemma~\ref{lemma:closedImage}, the image of $\varphi_\nu$, which we denote by $M_\nu$, is a projective scheme over $\C$. By a slight abuse of notation, we will denote by $\varphi_\nu$ also the morphism from $R$ to $M_\nu$. Then $M_\nu$ is endowed with an ample line bundle $\sheaf{A}_{\nu}$ such that ${\varphi}_\nu^*\sheaf{A}_{\nu} \cong \sheaf{L}^\nu$.

For $k > 0$, the $k$-th symmetric power $\Sym^k\Gamma(R,\sheaf{L}^{\nu})^{G}$ is a linear subsystem of $\Gamma(R,\sheaf{L}^{k\nu})^G$ that globally generates $\sheaf{L}^{k\nu}$. This gives a commutative diagram
\begin{center}
\begin{tikzcd}[column sep=small]
& R \arrow[dl,"\varphi_{k\nu}"'] \arrow[dr,"\varphi_{\nu}"] & \\
M_{k\nu} \arrow[rr,"\pi_{k}"] & & M_\nu
\end{tikzcd}
\end{center}
with $\pi_{k}$ finite such that $\pi_{k}^* \sheaf{A}_{\nu}^k = \sheaf{A}_{k\nu}$.

\begin{remark}\label{remark:HL}
In \cite[Prop.~8.2.6]{huybrechts2010geometry}, the authors considered the projective limit of a system $\{ M_\nu \}$ as above dominated by a $G$-scheme $R^{\mu ss}$. Even though this limit exists, in general there is no guarantee that the system eventually stabilizes or, equivalently, that the limit is a priori of finite type. We shall provide an argument for this fact in the proof of Theorem~\ref{thm:Unique}, as the scheme $R^{\mu ss}$ they consider is actually $G$-proper (see Lemma~\ref{lemma:RssGproper}).   
\end{remark}

It should become apparent by now that our approach shares some similarities with the one described by Lazarsfeld for the Iitaka fibration \cite[Ch.~2]{lazarsfeld2017positivity}, except that $R$ is neither projective nor normal. In this regard, we state the following theorem which is a generalized version of \cite[Thm.~2.1.27]{lazarsfeld2017positivity}.

\begin{theorem}\label{thm:Unique}
Under the same assumptions as above, there exists a unique projective scheme $M$ over $\C$ and a $G$-invariant surjective morphism $\varphi : R \to M$ such that for $\nu \gg 0$ we have $M_\nu = M$ and $\varphi_\nu = \varphi$. Moreover, $M$ is endowed with an ample line bundle $\sheaf{A}$ such that $\varphi^*\sheaf{A} \cong \sheaf{L}^e$, where $e$ is the exponent of $M(R,\sheaf{L})$, i.e. the largest natural number such that each member of $M(R,\sheaf{L})$ is a multiple of $e$.
\end{theorem}
\begin{proof}
Let $\nu \in M(R,\sheaf{L})$. We apply the result of Proposition~\ref{Prop:normalization} to the morphism $\varphi_\nu: R \to M_\nu$. Denote by $M$ the $G$-normalization of $M_\nu$ in $R$. This gives a natural factorization $R \xrightarrow{\varphi} {M} \xrightarrow{\pi_\nu} M_\nu$ of $\varphi$, with the following properties:
\begin{enumerate}[(i),nosep]
\item $M = \underline{\Spec}_{M_\nu}\!\left(({\varphi_{\nu}}_* \sheaf{O}_{R})^{G}\right)$ and $({\varphi}_* \sheaf{O}_{R})^{G} = \sheaf{O}_{{M}}$,
\item the morphism $\pi_\nu : {M} \to M_\nu$ is finite,
\item the morphism $\varphi$ is surjective and $M$ is projective over $\C$.
\end{enumerate}
By construction, the morphism $\varphi_\nu : R \to M_\nu$ satisfies ${\varphi_\nu}^* \sheaf{A}_{\nu} \cong \sheaf{L}^\nu$. Since $\pi_\nu$ is finite, $\sheaf{A}_{\nu}$ pulls back to an ample line bundle on $M$, denoted by $\sheaf{B}$. Thus, for $k > 0$, we have
\begin{align*}
\varphi^* \sheaf{B}^k \cong \varphi^*(\pi_\nu^*\sheaf{A}_{\nu}^k) \cong \varphi_\nu^*\sheaf{A}_{\nu}^k \cong \sheaf{L}^{k\nu}.
\end{align*}
Then
\begin{align}\label{eq:linSystem}
H^0(R, \sheaf{L}^{k\nu})^{G} \cong H^0(M,\varphi_*(\varphi^* \sheaf{B}^k)^{G}) \cong H^0(M, \sheaf{B}^k),
\end{align}
where the last isomorphism follows from (i) and by using the equivariant projection formula (see \cite[Lem.~4.9]{greb2017compact}). Indeed,
\begin{align*}
\varphi_* (\varphi^* \sheaf{B}^k \otimes \sheaf{O}_{R})^{G} \cong \sheaf{B}^k \otimes (\varphi_* \sheaf{O}_{R})^{G} \cong \sheaf{B}^k.
\end{align*}
For $k \gg 0$, the line bundle $\sheaf{B}^k$ is very ample on $M$, and so \eqref{eq:linSystem} implies that $\varphi$ is in fact defined by the linear system $H^0(R, \sheaf{L}^{k\nu})^{G}$. Thus $M_{k\nu} = M$ and $\varphi_{k\nu} = \varphi$ for $k \gg 0$.

The second part of the statement follows exactly as in the proof of \cite[Thm.~2.1.27]{lazarsfeld2017positivity} and for this reason we omit the details. 
\end{proof}
\section{A moduli space of slope-semistable sheaves}\label{section:Construction}

In this section, we will construct a moduli space of $\hat{\mu}$-semistable sheaves on $X$. As before, fix a class $c \in \grothnum{X}$ of dimension $d > 0$ and multiplicity $r$.  Denote by $P$ the Hilbert polynomial of $c$. 

It is known that the family of $\hat{\mu}$-semistable sheaves of class $c$ is bounded, cf. \cite[Thm.~3.3.7]{huybrechts2010geometry}, thus for $m \gg 0$ any such sheaf $E$ is $m$-regular (in the sense of Castelnuovo--Mumford) and there is a quotient
\begin{align*}
   \sheaf{O}_X(-m)^{\oplus P(m)} \to E \to 0
\end{align*}
with $P(m) = h^0(E(m))$. Set $V \coloneqq \C^{\oplus P(m)}$ and consider the subset $\Rss \subset \Quot(V \otimes \sheaf{O}_X(-m),P)$ of all quotients $[q: V \otimes \sheaf{O}_X(-m) \to E]$ such that
   \begin{enumerate}[noitemsep]
       \item $E$ is $\hat{\mu}$-semistable of class $c$ and
       \item the induced map $V \to H^0(E(m))$ is an isomorphism.
   \end{enumerate}
Then $\Rss$ is a locally closed subscheme and the linear group $G \coloneqq \SL{V}$ acts naturally on $\Rss$ by base change of $V$. Note that $\Rss$ is not necessarily proper. However, by using Langton's valuative criterion of properness for slope-semistable sheaves (see \cite{langton1975valuative}), we prove below that $\Rss$ is in fact $G$-proper (see Definition~\ref{def:Gproper}).
\begin{lemma}\label{lemma:RssGproper}
$\Rss$ is $G$-proper over $\C$.
\end{lemma}
\begin{proof}
Let $A$ be a discrete valuation ring over $\C$ of quotient field $K$ and $\xi_K : \Spec(K) \to \Rss$ a $K$-point, which corresponds to a quotient $[q_K : K \otimes V \to F_K] \in \Rss$. By Langton's result \cite[Thm.~2.B.1]{huybrechts2010geometry}, we can extend $F_K$ to an $A$-flat family $F_A$ of $\hat{\mu}$-semistable sheaves of class $c$ on $X$, which furthermore induces an $A$-point $\xi_A : \Spec(A) \to \Rss$ such that $\xi_A|_K$ and $\xi_K$ differ by a group element in $G(K)$. This completes the proof. 
\end{proof}

Let $\sheaf{F}$ denote the universal family of quotients over $\Rss$, and consider the determinant line bundle (see Section~\ref{section:LineBundles})
\begin{align*}
  \sheaf{L} \coloneqq \lambda_\sheaf{F}(w_{l',m'} \cdot h^{d-1})
\end{align*}
over $\Rss$. For $l' \gg m' \gg 0$ and $a \gg 0$, we may assume that $\sheaf{L}$ is $G$-semiample. Indeed, if $d = 2$, then this holds by Theorem~\ref{thm:semiAmpleness2}. Otherwise, for $d > 2$, in order to ensure that Theorem~\ref{thm:semiAmpleness} holds true and implicitly that $\sheaf{L}$ is $G$-semiample, we replace $\Rss$ by its weak normalization and we work exclusively over the category $\schCW$ of weakly normal varieties. 

\begin{thm}
There exists a unique triple $(\Mss, \sheaf{A}, e)$ formed of a weakly normal projective variety $\Mss$ endowed with an ample line bundle $\sheaf{A}$ and a natural number $e > 0$ such that there is a natural transformation $\Psi: \mf^{wn} \to \Hom_{\schCW}(-,\Mss)$, that associates to any weakly normal variety $S$ and any $S$-flat family $\sheaf{E}$ of $\hat{\mu}$-semistable sheaves of class $c$ on $X$ a classifying morphism $\Psi_\sheaf{E}: S \to \Mss$, satisfying the following properties:
    \begin{enumerate}[(1)]
\item For any $S$-flat family $\sheaf{E}$ of $\hat{\mu}$-semistable sheaves of class $c$ on $X$, the classifying morphism $\Psi_\sheaf{E}$ satisfies
\begin{align*}
    \Psi_\sheaf{E}^*(\sheaf{A}) \cong \lambda_\sheaf{E}(w_{l',m'} \cdot h^{d-1})^e,
\end{align*}
where $\lambda_\sheaf{E}(w_{l',m'} \cdot h^{d-1})$ is the determinant line bundle on $S$ defined in Section~\ref{section:LineBundles}.
\item For any other triple $(M',\sheaf{A}',e')$, with $M'$ a projective scheme over $\C$, $\sheaf{A}'$ an ample line bundle on $M'$ and $e'$ a natural number satisfying property (1), we have $e | e'$ and there exists a unique morphism $\phi: \Mss \to M'$ such that $\phi^*\sheaf{A}' \cong \sheaf{A}^{(e'/e)}$.
\end{enumerate}
Moreover, if the class $c$ has dimension $2$, then the same result as above holds true over the category $\schC$ of schemes of finite type over $\C$.
\end{thm}
\begin{proof}
As $\Rss$ is a $G$-proper scheme over $\C$ and $\sheaf{L}$ is $G$-semiample, we can consider the Iitaka $G$-fibration corresponding to $(\Rss,\sheaf{L})$ to obtain a triple $(\Mss, \sheaf{A}, e)$, as given by Theorem~\ref{thm:Unique}, together with a $G$-invariant morphism $\varphi : \Rss \to \Mss$. In fact, as $\Rss$ is weakly normal, then from the universal property of weak normality it follows that $\Mss$ is a weakly normal projective variety. Now let $S$ be a weakly normal variety and $\sheaf{E}$ an $S$-flat family of $\hat{\mu}$-semistable sheaves of class $c$ on $X$. By construction, any fiber $\sheaf{E}_s$ is $m$-regular, and so ${p_1}_*(\sheaf{E}(m))$ is a locally free $G$-equivariant $\sheaf{O}_{S}$-sheaf of rank $P(m)$. Denote by $\mathcal{S}$ the projective frame bundle associated to ${p_1}_*(\sheaf{E}(m))$ and let $\pi:\mathcal{S} \to S$ be the canonical projection. As in Section~\ref{subsect:semiamplenessTheorems}, we obtain a diagram
\begin{center}
\begin{tikzcd}[column sep=normal]
\mathcal{S} \ar[d,"\pi"] \ar[r,"\Phi_{\sheaf{E}}"] &  \Rss \ar[d,"\varphi"]  \\ 
 S & \Mss
\end{tikzcd}
\end{center}
with $\Phi_{\sheaf{E}}$ a $G$-equivariant morphism. As $\pi$ is a categorical quotient, the $G$-invariant map $\varphi \circ \Phi_{\sheaf{E}}$ descends to a $G$-invariant morphism $\Psi_\sheaf{E}: S \to \Mss$. Furthermore, as the determinant map $\lambda$ is functorial, we obtain
\begin{align*}
    \pi^* \Psi_\sheaf{E}^*(\sheaf{A}) \cong \Phi_{\sheaf{E}}^*(\sheaf{L}^e) \cong \lambda_{\pi^*\sheaf{E}}(w_{l',m'} \cdot h^{d-1})^e \cong \pi^*\lambda_\sheaf{E}(w_{l',m'} \cdot h^{d-1})^e.
\end{align*}
But $\pi^* : \Pic(S) \to \Pic(\mathcal{S})$ is injective (see \cite[Lem.~2.14]{le1992fibre}), and so
\begin{align*}
  \Psi_\sheaf{E}^*(\sheaf{A}) \cong \lambda_\sheaf{E}(w_{l',m'} \cdot h^{d-1})^e,
\end{align*}
which proves the first part of the theorem.

For the second part, consider a triple $(M',\sheaf{A}',e')$ as given in (2). Since $M'$ satisfies property (1), there exists a morphism $\varphi' : \Rss \to M'$ corresponding to the universal family of quotients over $\Rss$ such that $\varphi'^*(\sheaf{A}') \cong \sheaf{L}^{e'}$. For $k > 0$, this induces a morphism
\begin{align*}
  H^0(M',\sheaf{A}'^k) \to H^0(\Rss,\sheaf{L}^{ke'})^G.
\end{align*} 
For $k \gg 0$, $\sheaf{A}'^k$ is very ample, hence $\varphi'$ factors through the morphism $\varphi_{ke'} : \Rss \to \CP(H^0(\Rss,\sheaf{L}^{ke'})^G)$. Denote by $M_{ke'}$ the image of $\varphi_{ke'}$. We have shown in Theorem~\ref{thm:Unique} that $\Mss = M_{ke'}$ for $k \gg 0$. Thus we obtain a morphism $\phi : \Mss \to M'$ as claimed in (2).
\end{proof}

\subsection{Geometric points of the moduli space}

In this section we investigate the geometric separation of points in the moduli space $\Mss$. As in the case of Gieseker-semistability, one may show that a quotient $E \in \Rss$ and its graded object $\gr_{JH}(E)$ corresponding to a Jordan-H\"{o}lder filtration of $E$ are indistinguishable inside the moduli space. Indeed, it is not difficult to construct a family $\sheaf{E}$ of sheaves parametrized by $\A_\C^1$ such that $\sheaf{E}_0 \cong \gr_{JH}(E)$ and $\sheaf{E}_t \cong E$ for all $t \neq 0$. Then the classifying morphism $\Psi_\sheaf{E}: \A^1 \to \Mss$ is constant, and so $E$ and $\gr_{JH}(E)$ are mapped to the same point of $\Mss$. For this reason, it is enough to study the separation of polystable sheaves. 

If $E$ is a pure sheaf on $X$ and $E^H$ its reflexive hull (see Section~\ref{subsect:reflHull}), then there is a short exact sequence
\begin{align*}
    0 \to E \to E^H \to T \to 0
\end{align*}
such that $T$ is supported in codimension $\leq 2$ in $\Supp(E)$. We denote by $C_{d-2}(E)$ the $(d-2)$-dimensional cycle corresponding to $T$, i.e.
\begin{align*}
  C_{d-2}(E) \coloneqq \sum_j (\length_{Z_j}(T))\langle Z_j \rangle,
\end{align*}
where the $Z_j$ are the irreducible components of dimension $d-2$ of the support of $T$. 

In what follows, we denote by $[E] \in \Rss$ the closed point of $\Rss$ corresponding to a quotient sheaf $E$ on $X$. Also, $\varphi : \Rss \to \Mss$ is the classifying morphism given by the Main Theorem, corresponding to the universal family $\sheaf{F}$ of quotients over $\Rss$.

\begin{lemma}\label{lemma:sepR}
    Let $E_1$ and $E_2$ be two polystable quotient sheaves of $\Rss$ with non-isomorphic reflexive hulls. Then $\varphi$ separates $[E_1]$ and $[E_2]$ in $\Mss$.
\end{lemma}
\begin{proof}
Firstly, we assume that $E_1$ and $E_2$ have distinct support. Then, for a general complete intersection $X^{(d-1)} \subset X$ of $d-1$ divisors in $\Pi_a$, $E_1|_{X^{(d-1)}}$ and $E_2|_{X^{(d-1)}}$ still have distinct support, so they are not $S$-equivalent. By Theorem~\ref{thm:StableRestriction}, for $a \gg 0$ we may assume that $E_1|_{X^{(d-1)}}$ and $E_2|_{X^{(d-1)}}$ are $\hat{\mu}$-semistable. Then there exists a $G'$-invariant section $\sigma \in H^0(\Quot_{X^{(d-1)}},H_{l'}^{\nu'})^{G'}$ that separates $E_1|_{X^{(d-1)}}$ and $E_2|_{X^{(d-1)}}$. As in Section~4.2, this yields a section $\Gamma_{\sheaf{F}}(\sigma)$ that separates $[E_1]$ and $[E_2]$, which is exactly what we need.

Now, suppose that $E_1$ and $E_2$ have same support. Let $D \in \Pi_a$ be a divisor avoiding the associated points of $E_1$ and $E_2$, and consider the short exact sequence
\begin{align*}
    0 \to E_2^H(-D) \to E_2^H \to E_2^H|_D \to 0
\end{align*}
Applying the functor $\Hom(E_1^H,-)$ we obtain a long exact sequence
\begin{align*}
    0 \to \Hom(E_1^H,E_2^H) \to \Hom(E_1^H|_D,E_2^H|_D) \to \Ext^1(E_1^H,E_2^H(-D)) \to \ldots,
\end{align*}
where the first morphism is injective since the sheaves are pure. By using Serre duality and the local-to-global Ext spectral sequence, one obtains
\begin{align*}
  \Ext^1(E_1^H,E_2^H(-D)) \cong H^0(X,\Exts^{n-1}(E_1^H,E_2^H) \otimes \omega_X(D))
\end{align*}
for $a \gg 0$. As $E_2^H$ is pure and $S_2$, we have $\Exts^{n-1}(E_1^H,E_2^H) = 0$, and so
\begin{align*}
    \Hom(E_1^H, E_2^H) \to \Hom(E_1^H|_D, E_2^H|_D)
\end{align*}
is an isomorphism. One can choose a general $D \in \Pi_a$ such that $E_1^H|_D$ and $E_2^H|_D$ remain reflexive, cf. \cite[Cor.~1.1.14]{huybrechts2010geometry}. Thus, repeating the argument above, we obtain an isomorphism
\begin{align*}
    \Hom(E_1^H, E_2^H) \to \Hom(E_1^H|_{X^{(d-1)}}, E_2^H|_{X^{(d-1)}})
\end{align*}
for a general smooth complete intersection $X^{(d-1)}$. By assumption $E_1^H \ncong E_2^H$, and so $E_1^H|_{X^{(d-1)}} \ncong E_2^H|_{X^{(d-1)}}$. For $a \gg 0$, one can choose $X^{(d-1)}$ such that
\begin{enumerate}[(i),nolistsep]
\item the restricted sheaves $E_1|_{X^{(d-1)}}$ and $E_2|_{X^{(d-1)}}$ remain polystable, 
\item $E_1|_{X^{(d-1)}} \cong E_1^H|_{X^{(d-1)}}$ and $E_2|_{X^{(d-1)}} \cong E_2^H|_{X^{(d-1)}}$.
\end{enumerate}
Indeed, the first condition is a consequence of Theorem~\ref{thm:StableRestriction}, and the second one is satisfied by a sufficiently general complete intersection. Therefore, for such $X^{(d-1)}$, $E_1|_{X^{(d-1)}}$ and $E_2|_{X^{(d-1)}}$ are not $S$-equivalent. Now the result follows as in the first part of the proof.
\end{proof}

By using the result above one can show that $\Mss$ provides a compactification of the moduli space of $\hat{\mu}$-stable reflexive sheaves of class $c$ on $X$. To be precise, let $M^{Gss}$ be the Simpson moduli space of semistable sheaves of class $c$ on $X$, and denote by $M_{refl}^s \subset M^{Gss}$ the open locus of $\hat{\mu}$-stable reflexive sheaves. By the universal property of $M^{Gss}$, there is a natural morphism $\rho : M_{refl}^s \to \Mss$. We show below that $\rho$ is an open embedding. 

\begin{prop}\label{prop:compactification}
The natural morphism $\rho : M_{refl}^s \to \Mss$ is an open immersion. 
\end{prop}
\begin{proof}
By Lemma~\ref{lemma:sepR}, we know that $\rho$ is injective. It remains to show that $\rho$ is also unramified at each point of $M_{refl}^s$. So choose a closed point $[E] \in M_{refl}^s$ corresponding to a $\hat{\mu}$-stable reflexive sheaf $E$. As shown in Lemma~\ref{lemma:sepR}, for general $D \in \Pi_a$, we have $\Ext^1(E,E(-D)) = 0$ and an injection
\begin{align*}
  \Ext^1(E, E) \to \Ext^1(E|_D, E|_D).
\end{align*}
Thus, for a general complete intersection $X^{(d-1)} \subset X$ of $d-1$ divisors in $\Pi_a$, we obtain an injection
\begin{align}\label{eq:extinj}
  \Ext^1(E, E) \to \Ext^1(E|_{X^{(d-1)}}, E|_{X^{(d-1)}}).
\end{align}
Now, suppose that the restriction $E|_{X^{(d-1)}}$ remains $\hat{\mu}$-stable, which holds for a general divisor of large degree by Theorem~\ref{thm:StableRestriction}. Then we have a rational morphism $\theta : M_{refl}^s \dashrightarrow M_{X^{(d-1)}}$, where $M_{X^{(d-1)}}$ denotes the Simpson moduli space of semistable sheaves on ${X^{(d-1)}}$. By the injectivity of \eqref{eq:extinj}, it follows that $\theta$ is an open immersion locally around $[E] \in M_{refl}^s$. By using the following commutative diagram
\begin{center}
\begin{tikzcd}[column sep=normal]
M_{refl}^s \arrow[r,"\rho"]  \arrow[rd,dashed,"\theta"]
  & \Mss \arrow[d,dashed] \\
    & M_{X^{(d-1)}}
\end{tikzcd}
\end{center}
one deduces that $\rho$ is also an open immersion locally around $[E] \in M_{refl}^s$.
\end{proof}

\begin{lemma}\label{lemma:sep2}
Let $E_1$ and $E_2$ be two polystable quotient sheaves of $\Rss$ having the same reflexive hull $F$ such that $C_{d-2}(E_1) \neq C_{d-2}(E_2)$. Then $\varphi$ separates $[E_1]$ and $[E_2]$ in $\Mss$. 
\end{lemma}
\begin{proof}
Choose $d-1$ divisors $D_1, \ldots, D_{d-1} \in \Pi_a$ such that $X^{(d-1)} = D_1 \cap \ldots \cap D_{d-1}$ and $X^{(d-2)} \coloneqq D_1 \cap \ldots \cap D_{d-2}$ are complete intersections in $X$. If these divisors are chosen general enough, then
\begin{enumerate}[nosep]
  \item $E_1|_{X^{(d-2)}}$ and $E_2|_{X^{(d-2)}}$ remain polystable and have the same reflexive hull $F|_{X^{(d-2)}}$,
  \item $C_{0}(E_1|_{X^{(d-2)}}) \neq C_{0}(E_1|_{X^{(d-2)}})$.
\end{enumerate}
The first condition is satisfied by Theorem~\ref{thm:StableRestriction} and \cite[Cor.~1.1.14]{huybrechts2010geometry}, and the second one clearly holds for a sufficiently general complete intersection. Moreover, by Corollary~\ref{lemma:completeRestriction} and Lemma~\ref{lemma:Giso}, we have the following $G$-isomorphisms
\begin{align*}
    \sheaf{L} \cong \lambda_{\sheaf{F}|_{\Rss \times X^{(d-2)}}}(w_{l',m'}|_{X^{(d-2)}} \cdot h|_{X^{(d-2)}}) \cong \lambda_{\sheaf{F}|_{\Rss \times X^{(d-1)}}}(w_{l',m'}|_{X^{(d-1)}})
\end{align*}
over a $G$-stable open subset $V \subset \Rss$ containing $E_1|_{X^{(d-2)}}$ and $E_2|_{X^{(d-2)}}$, whose complement has codimension $\geq 2$ in $\Rss$. Thus, it is enough to find sections of $\lambda_{\sheaf{F}|_{\Rss \times X^{(d-2)}}}(w_{l',m'}|_{X^{(d-2)}} \cdot h|_{X^{(d-2)}})$ over $V$ that separate $E_1|_{X^{(d-2)}}$ and $E_2|_{X^{(d-2)}}$, since then we can extend them over the whole weakly normal variety $\Rss$, exactly as we did in the proof of the Semiampleness Theorem~\ref{thm:semiAmpleness}. In conclusion, we can restrict ourselves to the case $d = 2$, which we treat in detail below.

In the two-dimensional case, the proof follows closely the one given by Jun Li in \cite[Lem.~3.6]{li1993algebraic}. We may assume that there is a point $y_0 \in \Supp(F)$ such that $\length_{y_0}(C_0(E_1)) > \length_{y_0}(C_0(E_2)) \geq 0$. Choose $D_0 \in \Pi_a$ a smooth divisor passing through $y_0$ such that $E_1|_{D_0}$ and $E_2|_{D_0}$ are reflexive at $D_0 \setminus \{y_0\}$, and $F|_{D_0}$ is $\hat{\mu}$-semistable. Now take another smooth divisor $D_1 \in \Pi_a$ such that $D_0 \cap D_1$ is a smooth complete intersection that avoids $y_0$. By blowing up $X$ along $D_0 \cap D_1$ we get a smooth projective scheme $\widetilde{X}$ with a natural map $q: \widetilde{X} \to X$. Moreover, the linear system $D_t$ generated by $D_0$ and $D_1$ extends to a base-point-free pencil $\widetilde{D}_t$ on $\widetilde{X}$, which induces a morphism  $p: \widetilde{X} \to \CP^1$. Note that $\widetilde{D}_t$ is canonically isomorphic with $D_t$ via $q$ for all $t \in \CP^1$. For this reason, in what follows we freely identify $q^*E_1|_{\widetilde{D}_t}$, $q^*E_2|_{\widetilde{D}_t}$ and $q^*F|_{\widetilde{D}_t}$ with $E_1|_{D_t}$, $E_2|_{D_t}$ and $F|_{D_t}$ respectively. 

By the choice of the linear system, locally around $0 \in \CP^1$, we can find an open neighborhood $U \subset \CP^1$ such that any fiber $\widetilde{D}_t$ is smooth for $t \in U$, and $E_1|_{D_t}$ and $E_2|_{D_t}$ are $\hat{\mu}$-semistable reflexive sheaves for $0\neq t \in U$. We denote by $p_U : \widetilde{X}_U \to U$ the pullback of $p$ along the inclusion $U \subset \CP^1$. 

For all $t \in U$, the $F|_{D_t}$ have the same Hilbert polynomial, say $\widetilde{P}$. According to \cite[Thm.~4.3.7]{huybrechts2010geometry}, there exists a relative moduli space $M_{\widetilde{X}_U/U}(\widetilde{P}) \to U$ endowed with an ample line bundle $\sheaf{O}_{M_{\widetilde{X}_U/U}}(1)$, such that each fiber $M_{D_t}$ is the Simpson moduli space of semistable sheaves on $D_t$ with Hilbert polynomial $\widetilde{P}$. Since $F|_{D_0}$ is $\hat{\mu}$-semistable, there is a section $v_0 \in H^0(M_{D_t}, \sheaf{O}_{M_{D_t}}(k))$ for $k \gg 0$, nonvanishing at $F|_{D_0}$. By using the base change theorem and shrinking $U$ if necessary, we can extend $v_0$ to a section $v \in H^0(M_{\widetilde{X}_U/U}, \sheaf{O}_{M_{\widetilde{X}_U/U}}(k))$.

As in Section~4.2, for each $t \in U$, there exists a rational morphism $\psi_t : \Rss  \dashrightarrow M_{D_t}$ corresponding to the universal family $\sheaf{F}$ on $\Rss \times X$. Moreover, we can pullback the restricted section $v|_{M_{D_t}}$ via $\psi_t$ to obtain a canonical section $\widetilde{v}_t \in H^0(\Rss,\sheaf{L}^\nu)^G$ for $\nu \gg 0$. Since this construction is canonical, we can view the $\widetilde{v}_t$ as restrictions of a section $\widetilde{v} \in H^0(\Rss \times U, p_1^*\sheaf{L}^\nu)^G$, where $p_1 : \Rss \times U \to \Rss$ is the first projection. 

Recall that  $\sheaf{L} = \lambda_\sheaf{F}(u)$ on $\Rss$, for an appropriate Grothendieck class $u \in K(X)$. Denote by $\widetilde{u}$ the pullback of the class $u$ to $K(\widetilde{X})$ and by $\widetilde{\sheaf{F}} \coloneqq (1_{\Rss} \times q)^*\sheaf{F}$. Consider the determinant line bundle
\begin{align*}
    \sheaf{L}_{\Rss \times \CP^1} \coloneqq \det((1_{\Rss} \times p)_{!}( \widetilde{\sheaf{F}}\otimes p_2^*\widetilde{u})) 
\end{align*}
on $\Rss \times \CP^1$. For $i = 1,2$, the fiber of $\sheaf{L}_{\Rss \times \CP^1}$ over $[E_i] \in \Rss$ gives a line bundle
\begin{align*}
    \sheaf{L}_{i,\CP^1} \coloneqq \det(p_{!}(q^*E_i \otimes p_2^*\widetilde{u}))
\end{align*}
on $\CP^1$. By using the base change theorem, we get
\begin{align*}
    \sheaf{L}_{\Rss \times \CP^1} \cong p_1^*A \otimes p_2^*\sheaf{O}_{\CP^1}(\alpha),
\end{align*}
where $A$ is a line bundle $\Rss$, and $\sheaf{O}_{\CP^1}(\alpha)$ a line bundle on $\CP^1$. Thus, we have canonical isomorphisms
\begin{align*}
    \sheaf{L}_{i,\CP^1} \cong \sheaf{O}_{\CP^1}(\alpha).
\end{align*}
On the other hand, consider the line bundle $p_1^*\sheaf{L}$ on $\Rss \times U$. As above, the fiber of $p_1^*\sheaf{L}$ over $[E_i] \in \Rss$ gives a trivial line bundle $L_i$ on $U$. Thus, if we fix a trivialization of $\sheaf{O}_{\CP^1}(\alpha)$ over $U$, then we have the following canonical isomorphisms
\begin{align*}
    \gamma_i : \sheaf{L}_{i,\CP^1}|_U \to L_i.
\end{align*}
For $i=1,2$, we have the following short exact sequences on $\widetilde{X}_U$:
\begin{align*}
    0 \to q^*E_i \to q^*F \to T_i \to 0,
\end{align*}
where $q: \widetilde{X}_U \to X$ is the natural map, and $T_i$ are torsion sheaves on $\widetilde{X}_U$. This gives us isomorphisms
\begin{align}\label{eq:isoDuals}
    \phi_i : \det(p_{!}(q^*F \otimes p_2^*\widetilde{u})) \otimes \det(p_{!}(T_i \otimes p_2^*\widetilde{u}))^{-1} \to \sheaf{L}_{i, \CP^1}.
\end{align}
over $U$. Since $E_1|_{D_t} \cong E_2|_{D_t} \cong F|_{D_t}$ for $0 \neq t \in U$, $T_i$ is supported on $D_0$ and $\det(p_{!}(T_i \otimes p_2^*\widetilde{u}))^{-1}$ is isomorphic to $\sheaf{O}_{\CP^1}(l_i 0)|_U$, where $l_i = \length(T_i|_{D_0}) = \length_{y_0}(C_0(E_i))$. Note that $l_1 > l_2$ by assumption.

Denote by $\widetilde{v}^i \in H^0(U,L_i^\nu)$ the restriction of $\widetilde{v}$ to $[E_i] \times U$. Since $\gamma_1$ and $\gamma_2$ are canonical, we get that $\phi_1^*(\gamma_1^*(\widetilde{v}^1))$ and $\phi_2^*(\gamma_2^*(\widetilde{v}^2))$ agree over $U \setminus \{ 0 \}$ as sections of the line bundle $\det(p_{!}(q^*F \otimes p_2^*\widetilde{u}))$. Note that the extension of $\phi_1^*(\gamma_1^*(\widetilde{v}^1))$ over $U$ cannot be trivial, since we assumed that $v_0(F|_{D_0}) \neq 0$. It follows by \eqref{eq:isoDuals} that $\widetilde{v}^1$ and $\widetilde{v}^2$ have different vanishing order at $0 \in U$. Thus, for some $t_1, t_2 \in U$, we get
\begin{align*}
    \frac{\widetilde{v}^1(t_1)}{\widetilde{v}^2(t_1)} \neq \frac{\widetilde{v}^1(t_2)}{\widetilde{v}^2(t_2)},
\end{align*}
from which it follows that
\begin{align*}
    \frac{\widetilde{v}_{t_1}([E_1])}{\widetilde{v}_{t_2}([E_1])} \neq \frac{\widetilde{v}_{t_1}([E_2])}{\widetilde{v}_{t_2}([E_2])}.
\end{align*}
Taking $\lambda = \widetilde{v}_{t_1}([E_1])/\widetilde{v}_{t_2}([E_1])$, then $\sigma = \widetilde{v}_{t_1}- \lambda \widetilde{v}_{t_2}$ is a $G$-invariant section of $\sheaf{L}^\nu$ such that $\sigma([E_1]) = 0$ and $\sigma([E_2]) \neq 0$. Therefore $\varphi$ separates $[E_1]$ and $[E_2]$ in $\Mss$.
\end{proof}

It remains to treat the case when two polystable quotiens $E_1$ and $E_2$ of $\Rss$ have the same reflexive hull $E_1^H \cong E_2^H$ and $C_{d-2}(E_1) = C_{d-2}(E_2)$. Here one can prove the following:

\begin{lemma}\label{lemma:sameRefl}
Let $S$ be connected (weakly normal) $\C$-scheme of finite type and $\sheaf{E}$ an $S$-flat family of sheaves of class $c$ on $X$ such that each geometric fiber satisfies $\sheaf{E}_s^H = F$ and $C_{d-2}(\sheaf{E}_s) = C$, where $F$ is a fixed reflexive hull and $C$ a fixed $(d-2)$-dimensional cycle on $X$. Then the classifying morphism $\Psi_\sheaf{E} : S \to \Mss$, given by the Main Theorem, is constant. 
\end{lemma}
\begin{proof}
The proof is similar to that of \cite[Lem.~5.7]{greb2017compact} and we omit it.
\end{proof}

We summarize the results of this section in the following corollary:

\begin{cor}\label{cor:resultsep}
If $E_1$ and $E_2$ are two polystable quotient sheaves of $\Rss$ such that either
\begin{enumerate}[nolistsep]
    \item $E_1^H \ncong E_2^H$, or
    \item $E_1^H \cong E_2^H$ with $C_{d-2}(E_1) \neq C_{d-2}(E_2)$,
\end{enumerate}
then $\varphi$ separates $[E_1]$ and $[E_2]$ in $\Mss$. Otherwise, if $E_1^H \cong E_2^H$, $C_{d-2}(E_1) = C_{d-2}(E_2)$ and, furthermore, there is a connected curve passing through $[E_1]$ and $[E_2]$ in $\Rss$, then $\varphi$ maps $[E_1]$ and $[E_2]$ to the same point of $\Mss$.
\end{cor}

\bibliography{bib/main}

\Address

\end{document}